\documentclass[10pt,a4paper,reqno]{amsart}
\usepackage{amsthm}
\usepackage{amsmath}
\usepackage{amssymb}
\usepackage{graphicx,color}

\usepackage[font=small]{caption}
\usepackage{subcaption}

\usepackage{multirow}
\usepackage[shortlabels]{enumitem}

\usepackage{soul} 

\usepackage{cite}

\makeatletter
\theoremstyle{plain}
\newtheorem{theorem}{Theorem}
\theoremstyle{definition}
\theoremstyle{remark}
\newtheorem{remark}[theorem]{Remark}
\theoremstyle{plain}
\newtheorem{proposition}[theorem]{Proposition}
\theoremstyle{plain}
\newtheorem{lemma}[theorem]{Lemma}
\theoremstyle{plain}
\newtheorem{corollary}[theorem]{Corollary}
\theoremstyle{definition}

\usepackage{amsfonts}
\usepackage{mathrsfs}

\newtheorem*{question*}{\it{QUESTION}}

\theoremstyle{plain}

\newtheorem*{oq*}{Open Question}

\newcommand{\N}{\mathbb{N}}
\newcommand{\R}{{\mathbb{R}}}
\newcommand{\C}{{\mathbb{C}}}

\newcommand{\Z}{{\mathbb{Z}}}
\newcommand{\dd}{{\rm d}}
\newcommand{\ii}{{\rm i}}
\newcommand{\e}{{\rm e}}
\newcommand{\ess}{{\rm e}}

\newcommand{\rr}{\operatorname{Re}}
\newcommand{\im}{\operatorname{Im}}
\newcommand{\jj}[2]{J_{#1}(#2)}
\newcommand{\djj}[2]{J'_{#1}(#2)}

\newcommand{\hh}[2]{H^{(1)}_{#1}(#2)}
\newcommand{\dhh}[2]{(H^{(1)}_{#1})'(#2)}

\renewcommand{\Re}{\mathop\mathrm{Re}\nolimits}
\renewcommand{\Im}{\mathop\mathrm{Im}\nolimits}

\newcommand{\dist}{\mathop\mathrm{dist}\nolimits}

\newcommand{\bigO}[1]{\mathcal{O}{\left(#1\right)}} 

\definecolor{DarkGreen}{rgb}{0,0.5,0.1} 

\definecolor{Purple}{rgb}{0.6,0,1}

\makeatother

\counterwithin{theorem}{section}

\begin{document}

\title[]{Optimal Lieb--Thirring type inequalities for Schr\" odinger and Jacobi operators with complex potentials}

\author{Sabine B{\"o}gli}
\address[Sabine B{\"o}gli]{
	Department of Mathematical Sciences, Durham University, Upper Mountjoy, Stockton Road, Durham DH1 3LE, UK
}
\email{sabine.boegli@durham.ac.uk}

\author{Sukrid Petpradittha}
\address[Sukrid Petpradittha]{
	Department of Mathematical Sciences, Durham University, Upper Mountjoy, Stockton Road, Durham DH1 3LE, UK
}
\email{sukrid.petpradittha@durham.ac.uk}


\date{\today}

\begin{abstract}
	We prove optimal Lieb--Thirring type inequalities for Schr\"odinger and Jacobi operators with complex potentials. Our results bound eigenvalue power sums (Riesz means) by the $L^p$ norm of the potential, where in contrast to the self-adjoint case, each term needs to be weighted by a function of the ratio of the distance of the eigenvalue to the essential spectrum and the distance to the endpoint(s) thereof. Our Lieb--Thirring type bounds only hold for integrable weight functions. To prove optimality, we establish divergence estimates for non-integrable weight functions. The divergence rates exhibit a logarithmic or even polynomial gain compared to semiclassical methods (Weyl asymptotics) for real potentials.
\end{abstract}

\maketitle

\section{Introduction}

	The $d$-dimensional Schr\"{o}dinger operator in the Hilbert space $L^{2}(\R^{d})$ is defined by
	\[
		H_{V}:=-\Delta+V
	\]
	with a potential $V$. In the following, let $p$ depend on the dimension $d$ as follows:
	\begin{equation}\label{eq:LiebThCond}
		p\geq 1,\;\mbox{ if }d=1; \qquad 
		p>1,\;\mbox{ if }d=2;		 \qquad
		p\geq d/2,\;\mbox{ if }d\geq3.
	\end{equation}
	If a real-valued potential $V$ is sufficiently regular, namely $V\in L^{p}(\R^{d})$, then the essential spectrum  $\sigma_{\ess}(H_{V})$ is $[0,\infty)$ and the discrete spectrum $\sigma_{\dd}(H_{V})$ (isolated eigenvalues of finite algebraic multiplicities) consists of negative eigenvalues which can accumulate only at the point $0$, the bottom of the essential spectrum. The classical Lieb-Thirring inequality (after Lieb and Thirring \cite{lie-thi_prl75,lie-thi_91}) states that there exists a constant $C_{p,d}>0$ depending on $p$ and $d$ such that for all real-valued potentials $V\in L^{p}(\R^{d})$ 
	\begin{equation}\label{eq:LiebTh}
		\sum_{\lambda\in\sigma_{\dd}(H)}|\lambda|^{p-d/2}\leq C_{p,d}\int_{\R^{d}}|V(x)|^{p}\;\dd x,
	\end{equation}
	where in the sum we repeat each eigenvalue according to its (finite) algebraic multiplicity. 
	For more background material on self-adjoint Lieb--Thirring inqualities, see e.g. \cite{inbook,frank2022schrodinger, fra-sim-wei_cmp08,hundertmark2002lieb}.
	
For a complex-valued potential $V\in L^{p}(\R^d)$ we still have $\sigma_{\ess}(H_{V})=[0,\infty)$ but the behaviour of the discrete spectrum can be much more wild. For example, there can be non-zero accumulation points of the discrete spectrum \cite{bogli2017schrodinger,bogli2023counterexample}. 
This immediately implies that the inequality \eqref{eq:LiebTh} is false for general complex potentials $V\in L^p(\R^d)$.
	
	Recent years have seen a significant interest in Lieb--Thirring type inequalities for the complex-potential case, see e.g.\ \cite{bogli2017schrodinger,boegli2023improved,bogli2023counterexample,bogli2025lieb,bogli2021lieb,bor-gol-kup_09,chri-zin_lmp17,cuenin2022schrodinger,demuth2013eigenvalues,dub_14,fra_18,frank2006lieb,golinskii2007lieb,gol-kup_17,GolStep,han_11,hansmann2011inequalities,hun-lie-tho_atmp98,sam_14}.  
	Frank, Laptev, Lieb, and Seiringer \cite{frank2006lieb} proved that for given $p\geq d/2+1$ and $\tau>0$ there exists a constant $C_{p,d,\tau}>0$ such that for all (complex-valued) $V\in L^{p}(\R^{d})$,
	there exists a bound for all eigenvalues outside of a cone around the essential spectrum,
	\begin{equation}\label{eq:LT cone}
		\sum_{\substack{\lambda\in\sigma_{\dd}(H_{V}) \\ |\im\lambda|\geq\tau\rr\lambda}}|\lambda|^{p-d/2}\leq 
		C_{p,d,\tau}\int_{\mathbb{R}^d}|V(x)|^{p}\;\dd x,
	\end{equation}
	where $C_{p,d,\tau}=C_{p,d}\left(1+\frac{2}{\tau}\right)^{p}$ for a constant $C_{p,d}>0$.

	By averaging the bound \eqref{eq:LT cone} with respect to the parameter $\tau$, Demuth, Hansmann and Katriel~\cite{demuth2009discrete} obtained a bound involving \emph{all} eigenvalues, namely, for any $0<\kappa<1$,
	\begin{equation}\label{eq: Lieb for complex}
		\sum_{\lambda\in\sigma_\dd(H_{V})}\dfrac{\text{dist}(\lambda,[0,\infty))^{p+\kappa}}{|\lambda|^{d/2+\kappa}}\leq C_{p,d,\kappa}\int_{\R^d}|V(x)|^{p}\;\dd x,
	\end{equation}
	where $C_{p,d,\kappa}>0$ is a constant depending on $p,d$ and $\kappa$. In \cite{boegli2023improved}, B\"{o}gli  improved the latter Lieb-Thirring type inequalities. More precisely, given a continuous, non-increasing function $f:[0,\infty)\to(0,\infty)$, if
	\begin{equation}\label{eq:OptimalCond}
		\int_{0}^{\infty} f(t)\;\dd t<\infty,
	\end{equation}
	then there exists a constant $C_{p,d,f}>0$ such that for all $V\in L^{p}(\mathbb{R}^{d})$
	\begin{equation}\label{eq:newLTineq_S}
		\sum_{\lambda\in\sigma_\dd(H_{V})}\dfrac{\text{dist}(\lambda,[0,\infty))^{p}}{|\lambda|^{d/2}}f\left(-\log\left(\frac{\text{dist}(\lambda,[0,\infty))}{|\lambda|}\right)\right)\leq C_{p,d,f}\int_{\mathbb{R}^d}|V(x)|^{p}\;\dd x
	\end{equation}
	where $C_{p,d,f}=C_{p,d}\,\left(\int_{0}^{\infty} f(t)\;\dd t+f(0)\right)$ for an $f$-independent constant $C_{p,d}>0$.
Note that the inequality \eqref{eq: Lieb for complex} can be recovered by inserting the exponential function $f(t)={\e}^{-\kappa t}$ into the formula \eqref{eq:newLTineq_S}. We remark that the inequalities \eqref{eq: Lieb for complex} and \eqref{eq:newLTineq_S} are generalisations of the classical Lieb-Thirring inequalities for self-adjoint Schr\"{o}dinger operators as they reduce to \eqref{eq:LiebTh} for a real-valued potential because, in that case, $\dist(\lambda,[0,\infty))=|\lambda|$ for every discrete (negative) eigenvalue.

Define the ratio of the left- and right-hand sides of \eqref{eq:newLTineq_S} by
	\[
		{\rm Ratio}(V,f):=\left(\int_{\mathbb{R}^d}|V(x)|^{p}\;\dd x\right)^{-1}\sum_{\lambda\in\sigma_\dd(H_{V})}\dfrac{\text{dist}(\lambda,[0,\infty))^{p}}{|\lambda|^{d/2}}f\left(-\log\left(\frac{\text{dist}(\lambda,[0,\infty))}{|\lambda|}\right)\right).
	\]
	In dimension $d=1$, B\"{o}gli \cite{boegli2023improved} proved that the assumption \eqref{eq:OptimalCond} cannot be removed. Indeed, if the integral \eqref{eq:OptimalCond} is infinite, then $\sup_{0\neq V\in L^{p}(\R)}{\rm Ratio}(V,f)=\infty$ for any $p\geq 1$.
	More precisely, taking $V_h=\ii h \chi_{[-1,1]}$ with $\chi_{[-1,1]}$ the characteristic function of the interval $[-1,1]$, in the large-coupling limit $0<h\to \infty$
	we have the divergence rate
\begin{equation}\label{eq:rate}
 {\rm Ratio}(V_h,f) \gtrsim F(\varepsilon \log h) 
 \end{equation}
for any $0<\varepsilon<1$ where $F(x):=\int_0^x f(t)\,{\rm d}t$ (see the proof of \cite[Thm.~2.2]{boegli2023improved}).

In the first main result of this paper, we prove that the bound \eqref{eq:newLTineq_S} is optimal in dimensions $d\geq 2$ as well. More precisely, if the integral \eqref{eq:OptimalCond} is infinite, taking the  potential $V_h=\ii h \chi_{B_{1}(0)}$ with $\chi_{B_{1}(0)}$ the characteristic function of the open unit ball in $\R^d$, in the large-coupling limit $0<h\to \infty$
	we prove the divergence rate \eqref{eq:rate} for  any $0<\varepsilon<1$ (see Theorem~\ref{thm: ratio for decreasing func} for the precise statement with uniformity in $f$ and see also 
	Corollary~\ref{cor:ratio for decreasing func no w}).
	In the second main result (Theorem~\ref{thm: ratio for increasing func}), we prove similar divergence rates for \emph{non-decreasing} functions $f$ that satisfy a certain monotonicity assumption. 
	As an application, for $f(t)=\e^{\xi t}$ with $\xi\geq 0$, we obtain the divergence rate
	\[\left(\int_{\mathbb{R}^d}|V_h(x)|^{p}\;\dd x\right)^{-1}
			\sum_{\lambda\in\sigma_\dd(H_{V_h})}\dfrac{{\rm{dist}}(\lambda,[0,\infty))^{p-\xi}}{|\lambda|^{d/2-\xi}} \gtrsim 
			\begin{cases} h^{\varepsilon \xi} & \text{ if }\xi>0,\\
			 \log h & \text{ if } \xi=0 \end{cases} \]
			(see Corollaries~\ref{cor:ratio for increasing func no w},~\ref{cor:exponential growth} for the precise statement and also Remarks~\ref{rem:f=1},~\ref{rem:exp}). This answers a question posed by Cuenin and Frank \cite[Question~2]{CFopenproblem2024}; note that their (equivalent) formulation is the rescaled version of our strong-coupling limit ($h\to \infty$) to study the operators $-\hbar^{2}\Delta+V$ in the semiclassical limit ($\hbar=1/\sqrt{h}\to 0$).
			Our answer proves logarithmic (for $\xi=0$) and polynomial (for $\xi>0$) gain compared to the ideas from semiclassical analysis (or Weyl's law).
			Note that we restrict the parameter to $\xi\geq 0$ since for $\xi<0$ the function $f$ is integrable and we recover the Lieb--Thirring type inequality \eqref{eq: Lieb for complex} (with $\kappa=-\xi$).
			
As a further development, if $p>d/2$, for each $f$ for which the integral \eqref{eq:OptimalCond} is infinite, by taking the sum of potentials $V_h$ with supports sufficently far away from each other, we can construct a potential $V\in L^p(\R^d)$ such that ${\rm Ratio}(V,f)=\infty$ (see Theorem~\ref{thm:sum}). This is a more direct proof of the optimality of the Lieb--Thirring type inequality~\eqref{eq:newLTineq_S}.
	
	Finally, we also show that the $\tau$-dependence of the constant $C_{p,d,\tau}$ in \eqref{eq:LT cone} is sharp (see Theorem~\ref{thm:exponent outside a cone sharp}).

In Section~\ref{sect:jacobi} we prove analogous results for (one-dimensional) Jacobi operators.
	Let $J$ be a Jacobi operator in the Hilbert  space $\ell^{2}(\Z)$ acting on a complex sequence $u=\{u_{n}\}_{n\in\Z}\in \ell^{2}(\Z)$ as
\[
	(Ju)_{n}=a_{n-1}u_{n-1}+b_{n}u_{n}+c_{n}u_{n+1},\quad n\in\Z,
\]
where $\{a_{n}\}_{n\in\Z},\{b_{n}\}_{n\in\Z}$ and $\{c_{n}\}_{n\in\Z}$ are given bounded complex sequences. Then $J$ is a bounded operator and can be represented by the doubly-infinite tridiagonal matrix
\[
	J=
	\begin{pmatrix}
		\ddots & \ddots & \ddots &  & & &\\
		 & a_{-1} & b_{0} & c_{0} & & &\\
		 & & a_{0} & b_{1} & c_{1} & &\\
		 & & & a_{1} & b_{2} & c_{2} &\\
		 & & & & \ddots & \ddots & \ddots
	\end{pmatrix}
	.
\]
The free Jacobi operator $J_{0}$ is defined via the particular case $a_{n}\equiv 1,c_{n}\equiv 1$ and $b_{n}\equiv 0$, i.e. its action on $u$ is given by
\[
	(J_{0}u)_{n}=u_{n-1}+u_{n+1},\quad n\in\Z.
\]
In this circumstance, it is well-known that $\sigma(J_{0})=\sigma_{\ess}(J_{0})=[-2,2]$.
Let $v=\{v_{n}\}_{n\in\Z}$ be a sequence defined by setting
\[
	v_{n}:=\max\{|a_{n-1}-1|,|a_{n}-1|,|b_{n}|,|c_{n-1}-1|,|c_{n}-1|\},\quad n\in\Z.
\]
If $\lim_{|n|\to\infty} v_{n}=0$, then $J$ is a compact perturbation of $J_{0}$ and hence $\sigma_{\ess}(J)=[-2,2]$. Now, the discrete spectrum $\sigma_{\dd}(J)\subset\C\backslash[-2,2]$ consists of isolated eigenvalues of finite algebraic multiplicities that can possibly accumulate anywhere in $[-2,2]$.

For the special case $a_{n}=c_{n}>0$ and $b_{n}\in\R$, the Jacobi operator   $J$ is self-adjoint and the Lieb-Thirring inequalities due to Hundertmark and Simon \cite{hundertmark2002lieb} read that if $v\in \ell^{p}(\Z)$ for some $p\geq 1$, then
\begin{equation}\label{eq:L-T for self-adjoint Jacobi}
	\sum_{\lambda\in\sigma_{\dd}(J),\;\lambda>2} |\lambda-2|^{p-1/2}+\sum_{\lambda\in\sigma_{\dd}(J),\;\lambda<-2} |\lambda+2|^{p-1/2}
	\leq C_p\,\|v\|^{p}_{\ell^{p}}, 
\end{equation}
where $C_p>0$ is a constant depending on $p$ only.

For non-self-adjoint Jacobi operators, there exist Lieb-Thirring type inequalities outside a diamond-shaped sector in the complex plane. These results are due to  Golinskii and Kupin \cite[Thm.~1.5]{golinskii2007lieb} but we use
the formulation of Hansmann and Katriel \cite[Eq.~(8)]{hansmann2011inequalities}.
 For $0\leq\omega<\pi/2$ let us define two sectors  
\[
	\Phi^{\pm}_{\omega}:=\{\lambda\in\C : 2\mp\Re\lambda<\tan(\omega)|\Im\lambda|\}.
\]
Then, by \cite[Eq.~(8)]{hansmann2011inequalities}, for $p\geq 3/2$ there exists a constant $C_{p,\omega}>0$ such that for all $v\in \ell^{p}(\Z)$
\begin{equation}\label{eq:diamond ineq by G-K}
	\sum_{\lambda\in\sigma_{\dd}(J)\cap\Phi^{+}_{\omega}} |\lambda-2|^{p-1/2}+\sum_{\lambda\in\sigma_{\dd}(J)\cap\Phi^{-}_{\omega}} |\lambda+2|^{p-1/2}\leq C_{p,\omega}\|v\|^{p}_{\ell^{p}},
\end{equation}
where $C_{p,\omega}=C_p\,(1+2\tan(\omega))^{p}$ for a constant $C_p>0$.

Hansmann and Katriel  \cite[Thm.~2]{hansmann2011inequalities} used this inequality to prove a bound for \emph{all} eigenvalues. Namely,
for $p\geq 3/2$ and $0<\kappa<1$,  there exists a constant $C_{p,\kappa}>0$ such that for all $v\in \ell^{p}(\Z)$
	\begin{equation}\label{eq:H-K bd for Jacobi}
		\sum_{\lambda\in\sigma_{\dd}(J)} \frac{\dist(\lambda,[-2,2])^{p+\kappa}}{|\lambda^{2}-4|^{1/2+\kappa}}\leq C_{p,\tau}\|v\|^{p}_{\ell^{p}}.
	\end{equation}
Note that this inequality reduces to \eqref{eq:L-T for self-adjoint Jacobi} in the self-adjoint case.

As main results on Jacobi operators, we prove a stronger version of the estimate \eqref{eq:H-K bd for Jacobi} involving an integrable function $f$ (the analogue of \eqref{eq:newLTineq_S} for  Jacobi operators, see Theorem~\ref{thm:new L-T for non-self Jacobi}) and we prove optimality in the sense that if $f$ is not integrable, then no such bound holds and we establish explicit divergence rates 
(see Theorem~\ref{thm:sharpness thm} for non-increasing $f$ and Theorem~\ref{thm:sharpness thm for increasing} for non-decreasing $f$ satisfying a monotonicity assumption). 
Finally, we also show that the $\omega$-dependence of the constant $C_{p,\omega}$ in \eqref{eq:diamond ineq by G-K} is sharp (see Theorem~\ref{thm:sharpconstJacobi}).

	\subsection*{Notation}
The notation $\gtrsim$ ($\lesssim$)  means that the inequality $\geq$ ($\leq$) holds up to a multiplicative constant.
The notation $\ll$ ($\gg$) means that the ratio of the left-hand side to the right-hand side (the right-hand side to the left-hand side) converges to $0$ in the limit.
In most instances we display the involved constants and indicate their dependencies by subscripts (unless stated otherwise, for ease of notation).

	\section{Schr\"odinger operators}

	\subsection{Main results}
	
	In this section we study multidimensional Schr\"{o}dinger operators. 
	The proofs of the following main results will be given in Section~\ref{sect:proofsSchr}.
	
	First we show
	 that the estimate \eqref{eq:newLTineq_S} is optimal in the sense that if the function $f$ is not integrable, then $\sup_{0\neq V\in L^{p}(\R^{d})}{\rm Ratio}(V,f)$ is infinity. To this end, we consider potentials of the form $V_h=\ii h \chi_{B_1(0)}$ for $h>0$ and are interested in the strong-coupling limit $h\to\infty$.
	
	\begin{theorem}\label{thm: ratio for decreasing func}
		Let $d\geq 2$, $p$ satisfy \eqref{eq:LiebThCond} and let $0<\varepsilon<1$. Take a function $w:[0,\infty)\to [1,\infty)$ with $w(h)\to\infty$ as $h\to\infty$ (arbitrarily slowly). Then there exist $C_{p,d}>0$ and $h_{*}\geq 1$ such that for all continuous, non-increasing functions $f:[0,\infty)\to(0,\infty)$ with $\int_{0}^{\infty} f(t)\;\dd t=\infty$ and all $h\geq h_{*}$
		\begin{equation}\label{eq:Sup(V,f) lower bd decreasing}
			{\rm Ratio}(V_h,f)\geq C_{p,d}\left(\frac{F(\varepsilon\log h)}{w(h)}-f(0)w(h)\right),
		\end{equation}
		where $F(x):=\int_{0}^{x} f(t)\;\dd t$.
	\end{theorem}
	
	\begin{remark}\label{rem: ratio for decreasing func}
	We remark that even though \eqref{eq:newLTineq_S} requires $p\geq d/2+1$, here Theorem \ref{thm: ratio for decreasing func} does not.
	Note that the right-hand side of \eqref{eq:Sup(V,f) lower bd decreasing} is divergent whenever $w(h)$ diverges sufficiently slowly, for example when $(w(h))^{2}\ll F(\varepsilon\log h)$ as $h\to\infty$. 
	\end{remark}

We use the function $w$ to show the explicit (uniform) dependence on $f$. However, for a fixed function $f$, we can apply the estimate for a function $w$ that diverges arbitrarily slowly and thus obtain an improvement of~\eqref{eq:Sup(V,f) lower bd decreasing}.
	The next result concerns this improvement.
	
	\begin{corollary}\label{cor:ratio for decreasing func no w}
		Let $d\geq 2$, $p$ satisfy \eqref{eq:LiebThCond} and $0<\varepsilon<1$. Given a continuous, non-increasing function $f:[0,\infty)\to(0,\infty)$ with $\int_{0}^{\infty} f(t)\;\dd t=\infty$, there exist $C>0$ and $h_{*}\geq 1$ (both possibly $f$-dependent) such that for all $h\geq h_{*}$
		\begin{equation}\label{eq:no w}
			{\rm Ratio}(V_h,f)\geq CF(\varepsilon \log h).
		\end{equation}
	\end{corollary}
	
	Next we broaden the study of divergence rates of the ratios to get lower bounds for the class of (positive, continuous) \emph{non-decreasing} functions. In exchange, we require the monotonicity of the \emph{tail} of the function $f(\log t^{2})/t$. 
	\begin{theorem}\label{thm: ratio for increasing func}
		Let $d\geq 2$, $p$ satisfy \eqref{eq:LiebThCond} and let $0<\varepsilon<1\leq x_0$. Take $w:[0,\infty)\to [1,\infty)$ with $w(h)\to\infty$ as $h\to\infty$ (arbitrarily slowly). Then there exist $C_{p,d}>0$ and $h_{*}\geq 1$ such that for all $h\geq h_{*}$ and all continuous, non-decreasing functions $f:[0,\infty)\to (0,\infty)$ such that $f(\log t^{2})/t$ is monotonic for $t\geq x_{0}$ one has
		\begin{equation}\label{eq:Sup(V,f) lower bd increasing}
			{\rm Ratio}(V_h,f)\geq\frac{C_{p,d}}{w(h)}\left(F(\varepsilon\log h)-F\left(\frac{\varepsilon}{2}\log h\right)\right)\geq C_{p,d}\,\frac{\varepsilon f(0)}{2w(h)}\log h.
		\end{equation}
	\end{theorem}

	Again we get an improvement for a fixed $f$.

	\begin{corollary}\label{cor:ratio for increasing func no w}
	Let $d\geq 2$, $p$ satisfy \eqref{eq:LiebThCond} and let $0<\varepsilon<1\leq x_0$. Given a continuous, non-decreasing function $f:[0,\infty)\to (0,\infty)$ such that $f(\log t^{2})/t$ is monotonic for $t\geq x_{0}$, there exist $C>0$ and $h_{*}\geq 1$ (both possibly $f$-dependent) such that for all $h\geq h_{*}$ 
		\begin{equation}\label{eq:no w for increasing func}
			{\rm Ratio}(V_h,f)\geq C\left(F(\varepsilon\log h)-F\left(\frac{\varepsilon}{2}\log h\right)\right).
		\end{equation}
	\end{corollary}

	\begin{remark}\label{rem:f=1}
For the special case of a constant weight function $f\equiv 1$ in \eqref{eq:newLTineq_S}, the validity of a Lieb--Thirring type estimate was published as an open question by Demuth, Hansmann, and Katriel in \cite{demuth2013lieb}. The construction in \cite{bogli2021lieb} answered the question to the negative in dimension $d=1$, and in \cite{bogli2025lieb} the construction was generalised to higher dimensions, with the same class of potentials $V_h$ as studied in the present paper. Note that in \cite{bogli2025lieb}, a lower bound of the form
		${\rm Ratio}(V_h,f) \geq C_{p,d} (\log h)^{\varepsilon}$ was found where $0<\varepsilon<1$.
		The new results presented here are an improvement over these results because Corollary~\ref{cor:ratio for increasing func no w} for $f\equiv 1$ yields a divergence order of at least $\log h$.
	\end{remark}

To answer \cite[Question~2]{CFopenproblem2024}, we apply Theorem~\ref{thm: ratio for increasing func} to the exponential function $f(t)=\e^{\xi t}$ for $\xi>0$ (for $\xi=0$ we have $f\equiv 1$ which was discussed in the latter Remark).

	\begin{corollary}\label{cor:exponential growth}
		Let $d\geq 2$, $p$ satisfy \eqref{eq:LiebThCond} and let $0<\varepsilon<1$. Take a function $w:[0,\infty)\to [1,\infty)$ with $w(h)\to\infty$ as $h\to\infty$ (arbitrarily slowly). Then there exist $C_{p,d}>0$ and $h_{*}\geq 1$ such that for all $h\geq h_{*}$ and all $\xi>0$
		\begin{equation}\label{eq:Sup(V,f) lower bd for exponential}
			\left(\int_{\mathbb{R}^d}|V_h(x)|^{p}\;\dd x\right)^{-1}
			\sum_{\lambda\in\sigma_\dd(H_{V_h})}\dfrac{{\rm{dist}}(\lambda,[0,\infty))^{p-\xi}}{|\lambda|^{d/2-\xi}}\geq C_{p,d}\frac{h^{\varepsilon\xi}}{\xi w(h)}(1-h^{-\varepsilon\xi/2}).
		\end{equation}
	\end{corollary}
	
	\begin{remark}\label{rem:exp}
		With aid of the equation~\eqref{eq:no w for increasing func} for $f(t)={\e}^{\xi t}$, we get a divergence rate of at least $C h^{\varepsilon\xi}$ for a (possibly $\xi$-dependent) constant $C>0$.
	\end{remark}

The following result proves optimality of the Lieb--Thirring type inequality~\eqref{eq:newLTineq_S} in a more direct way.
\begin{theorem}\label{thm:sum}
Let $d\in\N$ and let $p$ satisfy \eqref{eq:LiebThCond} with $p>d/2$. Then, for every continuous, non-increasing function $f:[0,\infty)\to(0,\infty)$ with $\int_{0}^{\infty} f(t)\;\dd t=\infty$ there exists $V\in L^p(\R^d)$ such that
${\rm Ratio}(V,f)=\infty.$
\end{theorem}
\begin{remark}
It would be interesting to know whether the result continues to hold for $d\geq 3$ and $p=d/2$. The scaling argument that is used in the proof breaks down at this point.
\end{remark}

In dimension $d=1$, B\"{o}gli \cite[Thm~2.4]{boegli2023improved} proved that the $\tau$-dependence of the constant $C_{p,d,\tau}$ in \eqref{eq:LT cone}, i.e.\ the order $\tau^{-p}$ as $\tau\to 0$, is sharp. Here we prove sharpness in dimensions $d\geq 2$.

\begin{theorem}\label{thm:exponent outside a cone sharp}
	Let $d\geq 2$, $p$ satisfy \eqref{eq:LiebThCond} and let $\varphi:(0,\infty)\to(0,\infty)$ be a continuous function such that $\varphi(\tau)\ll\tau^{-p}$ as $\tau\to 0$. Then
	\begin{equation}\label{eq:LT cone exponent sharp}
		\limsup_{\tau\to 0^{+}}\sup_{0\neq V\in L^{p}(\R^{d})} 
		\left( \varphi(\tau)\int_{\mathbb{R}^d}|V(x)|^{p}\;\dd x\right)^{-1}
		\sum_{\substack{\lambda\in\sigma_{\dd}(H_{V}) \\ |\im\lambda|\geq\tau\rr\lambda}}
		|\lambda|^{p-d/2}=\infty.
	\end{equation}
\end{theorem}
The proof relies again on eigenvalue estimates for the class of potentials $V_h$ for $h>0$.

\subsection{Preliminaries}\label{Preliminaries}

	We devote this section to preparations for the proofs of our main results and use this opportunity to introduce our notation and terminology. The key ingredient is the asymptotics in \cite{bogli2025lieb} on eigenvalues $\lambda_{\ell,j}$, with error bounds that were shown to be uniform in two parameters (integers) $j, \ell$ in certain $h$-dependent index sets. The stronger results in the present paper require to work with larger index sets that depend on the function $w(h)$ used in Theorems~\ref{thm: ratio for decreasing func} and~\ref{thm: ratio for increasing func}.
	Hence, in the following we summarise the approach from \cite{bogli2025lieb} and show that the asymptotics continue to hold for the new index sets depending on $w(h)$, with adapted error bounds.  
	
	First of all, 
we note that since the potential $V_h$ is purely imaginary with non-negative imaginary part, a numerical range argument \cite[Lem.~2]{bogli2025lieb} shows that all eigenvalues $\lambda$ belong to the first quadrant of the complex plane  ($\Re \lambda\geq 0$ and $\im \lambda\geq 0$) and hence $\dist(\lambda,[0,\infty))=\im\lambda$.

Since the potential $V_h$ is spherically symmetric, we find solutions of the eigenvalue problem
$-\Delta f+V_{h}f=\lambda f $
by using spherical coordinates and solving a corresponding radial eigenvalue problem. 
To this end, we use complex parameters $k, m$ as follows: let $m\in\C$ with  $\Re m>0$ and set  $k:=\sqrt{\ii h+m^{2}}$ where we take the principal branch of the square root function.	
We assume that $\im k=\im \sqrt{\ii h+m^{2}} >0$.
For $\ell\in\N_{0}$, we  make the ansatz $f(x)=\psi(|x|)Y^{(\ell)}(x/|x|)$ where $Y^{(\ell)}$ is the spherical harmonic of degree $\ell$, defined on the $d$-dimensional unit sphere. 
Then, by \cite[Sect.~2.2]{bogli2025lieb}, $f\in L^{2}(\R^{d})$ is an eigenfunction corresponding to the eigenvalue $\lambda:=k^2=\ii h+m^{2}$ if $\psi\in L^2((0,\infty),r^{d-1}\,{\rm d}r)$ is the radial  ($r=|x|$) function defined by
	\[
		\psi(r)=	
		\begin{cases}
			H^{(1)}_\nu(k)r^{1-d/2}J_{\nu}(mr)\quad&\textrm{if\;\;} 0<r<1,\\
			J_\nu(m)r^{1-d/2}H^{(1)}_\nu(kr)\quad&\textrm{if\;\;} r\geq 1,
		\end{cases}
	\]
	and $m,k$ satisfy the characteristic equation
	\begin{equation}\label{eq:CharacteristicSign}
		\frac{k}{m}=\dfrac{\djj{\nu}{m}\hh{\nu}{k}}{\jj{\nu}{m}\dhh{\nu}{k}};
	\end{equation}
here $J_{\nu},H^{(1)}_{\nu}$ are respectively the Bessel and Hankel functions of the first kind of order 
\begin{equation}\label{eq:Nudefinition}
		\nu=\ell+\frac{d}{2}-1.
	\end{equation} 
	For the theory of Bessel functions and their classical asymptotics we refer to, for instance, \cite{abramowitz1948handbook,dlmf,olver1997asymptotics,watson1995treatise}.
Standard results on the Laplacian in spherical coordinates (see e.g. \cite{stein1971introduction} and also \cite[Thm.~3.49]{frank2022schrodinger}) imply that each eigenvalue $\lambda$ has the algebraic multiplicity at least
\begin{equation}\label{eq:eigenspace}
		\binom{d+\ell-1}{d-1}-\binom{d+\ell-3}{d-1}.
\end{equation}
	
The set-up in \cite[Sect.~4.1]{bogli2025lieb} introduced the constants $\alpha,\beta,\gamma$ and $\varepsilon$ with the following conditions
	\[
		0<\alpha<\beta<\gamma<\frac{1}{2}\qquad\text{and}\qquad 0<\varepsilon<1.
	\] 
	Here we make the modification to let $\alpha=\alpha(h)$ depend on  $h>0$ while still satisfying the above restrictions. 
	More precisely, we fix the parameters $0<\beta<\gamma<1/2$ and let $\alpha(h)\in (0,\beta)$ for all $h>0$. Assume further that $\alpha(h)$ converges to $0$ so slowly that $h^{-\alpha(h)}\to 0$ as $h\to\infty$. Now, let us take an arbitrary non-increasing function $g:(0,\infty)\to(0,1]$ such that $g(h)\to 0$ as $h\to\infty$ but so slowly that
	\begin{equation}\label{eq:gbound}
		g(h)\geq 2h^{\beta-\gamma}
	\end{equation}
	for all $h>1$.
	
	Instead of using the index sets $\mathcal L(h)$, $\mathcal J(h,\ell)$ in \cite[Eq.~(44), (46)]{bogli2025lieb}, we replace $\alpha$ by $\alpha(h)$ and define the sets
	\begin{equation}\label{eq:restr_ell}
		\mathcal{L}_{h}:=\left\{\ell\in\N : h^{\alpha(h)+1/2}\leq\ell\leq h^{\beta+1/2}\right\}
	\end{equation}
	for $h>0$, and
	\begin{equation}\label{eq:set_J}
		\mathcal{J}_{h,\ell}:=\left\{j\in\N : \frac{\ell}{g(h)}\leq j\leq h^{\gamma+1/2} \right\}
	\end{equation}
	for $h>0$ and $\ell\in\mathcal L_h$. Note that $g(h)$ replaces $\log^{-q}\ell$ in \cite[Eq.~(46)]{bogli2025lieb} and, by~\eqref{eq:gbound}, $\mathcal{J}_{h,\ell}\neq\emptyset$ when $h$ is large enough.
	
Let $j\in\N$, $\nu>0$ and $h>0$. We adopt the  auxiliary functions \cite[Eq.~(47)]{bogli2025lieb}
	\[
		f_{\nu,j}(z):=\theta_{\nu}(z)-\frac{\pi}{4}-2\pi j-\ii\log\frac{\sqrt{h}}{4\pi j},
	\]
	where $\theta_{\nu}(z)$ is the phase function given in terms of Bessel functions by 
	\[
		\theta_\nu(z)=\arctan \frac{Y_\nu(z)}{J_\nu(z)}
	\]
	with the standard branch satisfying $\theta_\nu(x)\to -\pi/2$ as $x\to 0^{+}$, see \cite[Sect.~3.2]{bogli2025lieb}.
	It was shown therein that there exists $A>0$ such that, for $\nu\geq1$, this branch of $\theta_{\nu}$ is an analytic function in the open convex set
\begin{equation}
 \mathcal{M}_{\nu}:=\{z\in\C :  A\nu<\Re z \mbox{ and } |z|<2\Re z\}.
 \label{eq:set_M}
\end{equation}
%
Following the arguments in \cite{bogli2025lieb}, we find asymptotics for the zeros of $f_{\nu,j}$ with error terms that are uniform in $j\in \mathcal J_{h,\ell}, \ell\in\mathcal L_h$. These asymptics give rise to asymptotic solutions of the characteristic equation \eqref{eq:CharacteristicSign}.
	
As in \cite[Eq.~(48)]{bogli2025lieb}, we define
	\[
		m_{\nu,j}^{(0)}:=2\pi j+\frac{\nu\pi}{2}+\frac{\pi}{2}+\ii\log\frac{\sqrt{h}}{4\pi j}.
	\]
	The following result is the analogue of \cite[Lem.~10]{bogli2025lieb} for the index sets $\mathcal L_h$, $\mathcal J_{h,\ell}$ in \eqref{eq:restr_ell}, \eqref{eq:set_J}.
	\begin{lemma}\label{lem:aux_zeros}
	Let $\nu=\ell+\frac{d}{2}-1$. Then there exists $h_{*}\geq 1$ such that for all $h\geq h_{*}$, all $\ell\in\mathcal{L}_h$ and all $j\in\mathcal{J}_{h, \ell}$, the following claims hold true:
\begin{enumerate}[(i)]
\item The function $f_{\nu,j}$ is analytic in the ball $B_{\nu}(m_{\nu,j}^{(0)})$ with a unique simple zero $m_{\nu,j}^{(1)}$ therein;
\item $|m_{\nu,j}^{(1)}-m_{\nu,j}^{(0)}|<\nu/2$;
\end{enumerate}
and, in addition, for any two indices $j_{1},j_{2}\in\mathcal{J}_{h,\ell}$, $j_1\neq j_2$, we have 
\begin{enumerate}
\item[(iii)] $|m_{\nu,j_1}^{(1)}-m_{\nu,j_2}^{(1)}|>4$.
\end{enumerate}
	\end{lemma}
	\begin{proof}
First we show that there exists a constant $C>0$ such that 
\begin{equation}\label{eq:unif_lim_inproof}
			\sup\left\{\frac{\nu}{\Re m} : |m-m_{\nu,j}^{(0)}|\leq\nu, \ell\in\mathcal{L}_{h}, j\in\mathcal{J}_{h,\ell} \right\}\leq Cg(h)
		\end{equation}
		for all $h$ sufficiently large.
To this end, we use~\eqref{eq:set_J} to estimate
\[
	\frac{\Re m_{\nu,j}^{(0)}}{\nu}=\frac{2\pi j}{\nu}+\frac{\pi}{2}+\frac{\pi}{2\nu}\geq\frac{j}{\nu}\gtrsim\frac{j}{\ell}\geq \frac{1}{g(h)}
\]
for all $h$ sufficiently large, where the non-displayed constant is independent of the choices of $\ell\in\mathcal{L}_h$ and $j\in\mathcal{J}_{h, \ell}$. Hence, for any $m$ in the closure of the ball $B_{\nu}(m_{\nu,j}^{(0)})$, we get
\[
\frac{\Re m}{\nu}\geq \frac{\Re m_{\nu,j}^{(0)}}{\nu}-1\gtrsim \frac{1}{g(h)},
\]
which implies $\Re m>0$ and \eqref{eq:unif_lim_inproof}.

Next we show that,  for all $h$ sufficienly large, $B_{\nu}(m_{\nu,j}^{(0)})\subset\mathcal{M}_{\nu}$ for all $\ell\in\mathcal{L}_{h}$ and all $ j\in\mathcal{J}_{h,\ell}$, where $\mathcal M_\nu$ is as in \eqref{eq:set_M}.
First, it follows readily from~\eqref{eq:unif_lim_inproof} that, for all $h$ sufficiently large, we have $A\nu<\Re m$ for all $m\in B_{\nu}(m_{\nu,j}^{(0)})$. Second, we estimate with \eqref{eq:set_J},
\[
 |\Im m_{\nu,j}^{(0)}|=\log\frac{4\pi j}{\sqrt{h}}\lesssim\log h,
\]
and taking also~\eqref{eq:restr_ell} into account, we get
\[
 \frac{|\Im m_{\nu,j}^{(0)}|}{\nu}\lesssim\frac{\log h}{h^{\alpha(h)+1/2}}\lesssim 1.
\]
As a result, for $m\in B_{\nu}(m_{\nu,j}^{(0)})$, we deduce that 
\[
 \frac{|\Im m|}{\Re m}\leq\frac{1+|\Im(m_{\nu,j}^{(0)})|/\nu}{-1+\Re(m_{\nu,j}^{(0)})/\nu}\lesssim g(h),
\]
which implies that $|m|<2\Re m$ for all $h$ sufficiently large. 
Thus $m\in \mathcal M_\nu$ and $B_{\nu}(m_{\nu,j}^{(0)})\subset\mathcal{M}_{\nu}$.
		
The rest of the proof is analogous to the one of \cite[Lem.~10]{bogli2025lieb}, with \eqref{eq:unif_lim_inproof} used instead of \cite[Eq.~(49)]{bogli2025lieb}.
	\end{proof}
	
	\begin{remark}\label{rem:B2_in_M}
We may always suppose that $h_{*}$ is large enough so that for all $h\geq h_{*}$ we have
\begin{equation}\label{eq:B2_in_M}
B_{2}(m_{\nu,j}^{(1)})\subset B_{\nu}(m_{\nu,j}^{(0)})\subset\mathcal{M}_{\nu}
\end{equation}
and
\begin{equation}\label{eq:B2_disjoint}
B_{2}(m_{\nu,j_1}^{(1)})\cap B_{2}(m_{\nu,j_2}^{(1)})=\emptyset
\end{equation}
for any $j,j_1,j_2\in\mathcal{J}_{h,\ell}$, $j_1\neq j_2$, and $\ell\in\mathcal{L}_h$; c.f.\ \cite[Eq.~(51), (52)]{bogli2025lieb}.
	\end{remark}

The next result is the analogue of \cite[Lem.~11]{bogli2025lieb} for the index sets $\mathcal L_h$, $\mathcal J_{h,\ell}$ in \eqref{eq:restr_ell}, \eqref{eq:set_J}.
	
	\begin{lemma}\label{lem:sup_re-im}
Let $\nu=\ell+\frac{d}{2}-1$.  Then there exist constants $C>0$, $\tilde{C}>0$ and $h_{*}\geq 1$ such that for all $h\geq h_{*}$ the following formulas hold:
		\begin{equation}\label{eq:sup_re_lim}
			\sup\left\{\left|\frac{\Re m}{2\pi j}-1\right| : |m-m_{\nu,j}^{(1)}|\leq2, \ell\in\mathcal{L}_{h}, j\in\mathcal{J}_{h,\ell}\right\}\leq Cg(h),
		\end{equation}
		and
		\begin{equation}\label{eq:sup_im_lim}
			\begin{aligned}
				\sup\left\{\left|\frac{\Im m}{\log\left(\sqrt{h}/(4\pi j)\right)}-1\right| : |m-m_{\nu,j}^{(1)}|\leq2, \ell\in\mathcal{L}_{h}, j\in\mathcal{J}_{h,\ell} \right\}
				\leq \tilde{C}\;\frac{1}{|\log g(h)|}.
			\end{aligned}
		\end{equation}
	\end{lemma}

	\begin{proof}
 From~\eqref{eq:restr_ell}, \eqref{eq:set_J}, one may notice that for $h\to\infty$
 \[
  	\frac{1}{j}\leq g(h)\to0 \quad\mbox{ and }\quad
  	\left| \frac{1}{\log\frac{\sqrt{h}}{4\pi j}} \right|=\frac{1}{\log\frac{4\pi j}{\sqrt{h}}}\lesssim 
  	\frac{1}{\log h^{\alpha(h)}+|\log g(h)|}
  	\to0,
 \]
 where the hidden constant is uniform in $j\in\mathcal{J}_{h, \ell}$ and $\ell\in\mathcal{L}_h$, therefore, bearing the triangle inequality in mind, it is sufficient to prove the uniform estimates
\[
			\sup\left\{\left|\frac{\Re m_{\nu,j}^{(1)}}{2\pi j}-1\right| : \ell\in\mathcal{L}_{h}, j\in\mathcal{J}_{h,\ell}\right\}\lesssim g(h),
		\]
and
		\[
			\sup\left\{\left|\frac{\Im m_{\nu,j}^{(1)}}{\log\left(\sqrt{h}/(4\pi j)\right)}-1\right| : \ell\in\mathcal{L}_{h}, j\in\mathcal{J}_{h,\ell} \right\}\lesssim g^{2}(h).
		\]
To this end, we proceed analogously as in the proof of  \cite[Lem.~11]{bogli2025lieb}.
Indeed, the first estimate follows from \cite[Eq.~(57)]{bogli2025lieb} and using that, by \eqref{eq:Nudefinition} and \eqref{eq:set_J},
		\[
			\frac{\nu}{j}\lesssim\frac{\ell}{j}\leq g(h).
		\]
To prove the second estimate, we  use \cite[Eq.~(58), (38)]{bogli2025lieb} to obtain
		\[
			\left|\frac{\log\left(\sqrt{h}/(4\pi j)\right)}{\im m_{\nu,j}^{(1)}}-1\right|
			\leq 4A^{2}\frac{\nu^{2}}{|m_{\nu,j}^{(1)}|^{2}}
			\leq 4A^{2}\frac{\nu^{2}}{(2\pi j)^{2}}\cdot\frac{(2\pi j)^{2}}{(\Re m_{\nu,j}^{(1)})^{2}}.
		\]
We estimate the right-hand side by using \eqref{eq:set_J} and \eqref{eq:sup_re_lim} to obtain an upper bound of order $g^{2}(h)$. Since $g^{2}(h)\lesssim 1/|\log g(h)|$ as $h\to\infty$, this proves the claim.
	\end{proof}

	\begin{remark}\label{rem:fourth-quadrant}
		One can infer from Lemma \ref{lem:sup_re-im} that if $h\geq h_{*}$, then the closure of the ball $B_{2}(m_{\nu,j}^{(1)})$ lies entirely in the fourth quadrant of the complex plane ($\Re m>0$ and $\im m<0$) for all $\ell\in\mathcal{L}_{h}$ and $j\in\mathcal{J}_{h,\ell}$.
	\end{remark}

Next we employ the error function in \cite[Eq.~(73)]{bogli2025lieb},
	\[
		\xi_{\nu}(m):=\left[\tan^{2}\theta_{\nu}(m)+\left(\frac{J_{\nu}'(m)H_{\nu}^{(1)}(k)}{J_{\nu}(m)(H_{\nu}^{(1)})'(k)}\right)^{\!2}\,\right]\cos^{2}\theta_{\nu}(m),
	\]
	where $k=\sqrt{\ii h+m^{2}}$. Here the principal branch of the square root is assumed.
The next result is the analogue of \cite[Lem.~16]{bogli2025lieb} for the index sets $\mathcal L_h$, $\mathcal J_{h,\ell}$ in \eqref{eq:restr_ell}, \eqref{eq:set_J}.

\begin{lemma}\label{lem:xi_prop}
Let $\nu=\ell+\frac{d}{2}-1$. Then there exists $h_{*}\geq 1$ such that for all $h\geq h_{*}$, $\ell\in\mathcal{L}_h$, and $j\in\mathcal{J}_{h, \ell}$, the function $\xi_{\nu}$ is analytic in $B_{2}(m_{\nu,j}^{(1)})$ and there is a~constant $C>0$ independent of $j$, $\ell$, and $m$ such that 
	\begin{equation}\label{eq:xi_small}
		|\xi_{\nu}(m)|\leq C(h^{\gamma-1/2}+g^{2}(h))
	\end{equation}
 for any $m\in B_{2}(m_{\nu,j}^{(1)})$.
\end{lemma}
	
\begin{proof}	
First we note that \cite[Lem.~14, 15]{bogli2025lieb} continue to hold for the new index sets $\mathcal L_h$, $\mathcal J_{h,\ell}$ in \eqref{eq:restr_ell}, \eqref{eq:set_J}, with $\alpha$ replaced by $\alpha(h)$ everywhere;
in particular
\begin{equation}\label{eq:Imtheta}
\Im \theta_\nu(m)\leq -\alpha(h) \log h.
\end{equation} 
In the proofs of both results, we use  \eqref{eq:unif_lim_inproof} instead of \cite[Eq.~(49)]{bogli2025lieb}.
The proof of the analogue of \cite[Lem.~15]{bogli2025lieb} needs an updated version of the two-sided estimate of $\log\frac{\sqrt h}{4\pi j}$; indeed, the upper bound needs to be replaced by
\begin{equation}\label{eq:quotientlog}
\begin{aligned}
\log\frac{\sqrt h}{4\pi j}&=\frac 1 2 \log h-\log 4\pi-\log j\leq \frac 1 2 \log h-\log\ell+\log g(h)\\
&\leq \log(h^{-\alpha(h)})+\log g(h)\leq \log(h^{-\alpha(h)})=-\alpha(h)\log h,
\end{aligned}
\end{equation}
where we have used  \eqref{eq:set_J}, \eqref{eq:restr_ell} and $g(h)\leq 1$.

Now we proceed as in the proof of \cite[Lem.~16]{bogli2025lieb} to arrive at
\[
 \xi_{\nu}(m)=\bigO{\frac{j^{2}}{mh}}+\bigO{\frac{j^{2}\nu^{2}}{m^{4}}},
\]
where the involved constants in the Landau symbols $\mathcal{O}$ are uniform in $\ell\in\mathcal{L}_h$ and $j\in\mathcal{J}_{h, \ell}$.
The first error term can be estimated, with~\eqref{eq:set_J} and~\eqref{eq:sup_re_lim}, as
\[
\frac{j^{2}}{|m|h}\lesssim\frac{j}{h}\leq h^{\gamma-1/2};
\]
note that $\gamma<1/2$. The second error term uses
\[
\frac{j^{2}\nu^{2}}{|m|^{4}}\lesssim\frac{\nu^{2}}{|m|^{2}}\lesssim g^2(h),
\]
where where we have used~\eqref{eq:unif_lim_inproof}. This completes the proof of \eqref{eq:xi_small}.
	\end{proof}

Next we move towards proving existence of solutions of the characteristic equation~\eqref{eq:CharacteristicSign}. To this end, we will need the following result which is the analogue of \cite[Lem.~17]{bogli2025lieb} for the index sets $\mathcal L_h$, $\mathcal J_{h,\ell}$ in \eqref{eq:restr_ell}, \eqref{eq:set_J}.

\begin{lemma}\label{lem:aux_eq}
Let $\nu=\ell+\frac{d}{2}-1$. Then there exists $h_{*}\geq 1$ such that for all $h\geq h_{*}$, $\ell\in\mathcal{L}_h$, and $j\in\mathcal{J}_{h, \ell}$, the following claims hold:
\begin{enumerate}[(i)]
\item The function
\[
 \operatorname{err}_{\nu,j}(m):=-1+\frac{m}{4\pi j}\frac{\e^{\ii\theta_{\nu}(m)}}{\cos\theta_{\nu}(m)}\sqrt{1-\xi_{\nu}(m)}
\]
is analytic in $B_{2}(m_{\nu,j}^{(1)})$ and there is a~constant $C>0$ independent of $j$, $\ell$ and $m$ such that 
 \begin{equation}\label{eq:err_small}
  |\operatorname{err}_{\nu,j}(m)|\leq C (h^{-2\alpha(h)}+g(h))
 \end{equation}
 for any $m\in B_{2}(m_{\nu,j}^{(1)})$.
\item If $m\in B_{2}(m_{\nu,j}^{(1)})$ satisfies
\begin{equation}
 \ii\left(\theta_{\nu}(m)-\frac{\pi}{4}-2\pi j\right)=\log\frac{4\pi j}{\sqrt{h}}+\log\left(1+\operatorname{err}_{\nu,j}(m)\right),
\label{eq:tow_char_eq}
\end{equation}
then $m$ is a solution of the characteristic equation~\eqref{eq:CharacteristicSign} with the corresponding $k=k(m)=\sqrt{\ii h+m^{2}}$.
\end{enumerate}
\end{lemma}

\begin{proof}
First we prove \eqref{eq:err_small}; then the remaining claims follow in exactly the same way as in the proof of \cite[Lem.~17]{bogli2025lieb}.
We estimate the three factors in
\begin{equation}
  \operatorname{err}_{\nu,j}(m)=-1+\frac{m}{2\pi j}\frac{1}{1+\e^{-2\ii\theta_{\nu}(m)}}\sqrt{1-\xi_{\nu}(m)}.
\label{eq:err_expre_inproof}
\end{equation}
First, with the aid of~\eqref{eq:xi_small}, we deduce that 
\[
\sqrt{1-\xi_{\nu}(m)}=1+\bigO{h^{\gamma-1/2}+g^2(h)}.
\]
Second, it follows from~\eqref{eq:Imtheta} that
\[
 	\frac{1}{1+\e^{-2\ii\theta_{\nu}(m)}}=\frac{1}{1+\bigO{h^{-2\alpha(h)}}}=1+\bigO{h^{-2\alpha(h)}}.
\]
Third, we use
\[
		\left|\frac{m}{2\pi j}-1\right|\leq\left|\frac{\Re m}{2\pi j}-1\right|+\left|\frac{\im m}{\log\left(\sqrt{h}/(4\pi j)\right)}\cdot\frac{\log\left(\sqrt{h}/(4\pi j)\right)}{2\pi j}\right|.
	\]
Note that \eqref{eq:quotientlog}, \eqref{eq:set_J}, \eqref{eq:restr_ell} yield
\begin{equation}\label{eq:quotientjh}
	\left|\frac{\log\left(\sqrt{h}/(4\pi j)\right)}{2\pi j}\right|
	= \frac{\log\left(4\pi j/\sqrt{h}\right)}{2\pi j} 
	\leq \frac{\log(4\pi h^{\gamma}) g(h)}{2\pi h^{\alpha(h)+1/2}} 	 
	\lesssim g(h)h^{\beta-\gamma};
\end{equation}
where in the last estimate we have used $0<\gamma-\beta<1/2$.
	Therefore, together with Lemma~\ref{lem:sup_re-im}, we obtain
	\[
 		\frac{m}{2\pi j}=1+\bigO{g(h)}.
	\]
Inserting the estimates into~\eqref{eq:err_expre_inproof}, also bearing~\eqref{eq:gbound} in mind, amounts to the uniform asymptotic formula~\eqref{eq:err_small}.
\end{proof}	

	
	Now, everything is in place for
	\begin{proposition}\label{prop:evls_ineq}
		Suppose that $d\geq2$ and let  $\nu=\ell+\frac{d}{2}-1$. Then there exists $h_{*}\geq 1$ such that for all $h\geq h_{*}$, all  $\ell\in\mathcal{L}_{h}$ and all $j\in\mathcal{J}_{h,\ell}$, the following claims are valid:
		\begin{enumerate}[(i)]
			\item There is a unique solution $m_{\nu,j}$ of the characteristic equation \eqref{eq:CharacteristicSign} inside $B_{2}(m_{\nu,j}^{(1)})$, with $m_{\nu,j_1}\neq m_{\nu,j_2}$ for two indices
$j_1,j_2\in \mathcal J_{h,\ell}$, $j_1\neq j_2$.
			\item The number
			\[
				\lambda_{\ell,j}:=\ii h+m_{\nu,j}^{2}
			\]
			is an eigenvalue of $H_{V_{h}}$ of algebraic multiplicity $m(\lambda_{\ell,j})$ at least as in ~\eqref{eq:eigenspace}.
			\item The real and imaginary parts of the eigenvalue satisfy
			\begin{equation}\label{eq:lam_ineq}
				\frac{h}{2}\leq\Im\lambda_{\ell,j}\leq h 
				\quad\mbox{ and }\quad
				(\pi j)^{2}\leq|\lambda_{\ell,j}|\leq (4\pi j)^{2}.  
			\end{equation}
		\end{enumerate}
	\end{proposition}
	
	\begin{proof}
		
		\emph{Proof of the claim $(i)$:} 
		This is the analogue of \cite[Prop.~18, Rem.~19]{bogli2025lieb} and is proved analogously, using \eqref{eq:unif_lim_inproof} instead of \cite[Eq.~(49)]{bogli2025lieb}. 
%
		
		\emph{Proof of the claim $(ii)$:} According to the eigenvalue construction in the beginning of the  Preliminaries, we choose only solutions $m_{\nu,j}$ amongst those found in the claim $(i)$ which satisfy the restrictions
		\[
			\Re m_{\nu,j}>0\qquad\text{and}\qquad\im\sqrt{\ii h+m_{\nu,j}^{2}}>0.
		\]
		By Remark \ref{rem:fourth-quadrant}, the first restriction is satisfied for all zeros $m_{\nu,j}$. Now, we consider the second restriction. Since we use the principal branch of the square root, $\im\sqrt{\ii h+m_{\nu,j}^{2}}>0$ if
		\[
			\im(\ii h+m_{\nu,j}^{2})=h+2\Re m_{\nu,j}\im m_{\nu,j}>0.
		\]
		Consequently, we restrict ourselves to solutions $m_{\nu,j}$ in the claim $(i)$ for which the latter condition is satisfied. It follows from \eqref{eq:sup_re_lim} and \eqref{eq:sup_im_lim} that there exist $C>0$ and $h_{0}\geq 1$ such that for all $h\geq h_{0}, \ell\in\mathcal{L}_{h}$ and $j\in\mathcal{J}_{h,\ell}$,
		\[
			|\Re m_{\nu,j}\im m_{\nu,j}|\leq Cj\log\frac{4\pi j}{\sqrt{h}}\lesssim h^{\gamma+1/2}\log h\ll h,
		\]
		where we have used \eqref{eq:set_J} and the fact that $\gamma<1/2$. Hence, there exists $h_{1}\geq h_{0}$ such that for all $h\geq h_{1}$, all $\ell\in\mathcal{L}_{h}$ and all $j\in\mathcal{J}_{h,\ell}$,
		\[
			\im(\ii h+m_{\nu,j}^{2})=h+2\Re m_{\nu,j}\im m_{\nu,j}\geq h-2|\Re m_{\nu,j}\im m_{\nu,j}|>0;
		\]
		hence  the zeros $m_{\nu,j}$ give rise to eigenvalues $\lambda_{\ell,j}$ of $H_{V_{h}}$ of the form $\lambda_{\ell,j}:=\ii h+m_{\nu,j}^{2}$.
		
		\emph{Proof of the claim $(iii)$:} First we prove the two-sided bound of the imaginary parts of the  eigenvalues. For all $h\geq h_{1}, \ell\in\mathcal{L}_{h}$ and $j\in\mathcal{J}_{h,\ell}$,
		\[
			\im\lambda_{\ell,j}=h+2\Re m_{\nu,j}\im m_{\nu,j}.
		\]
		It follows from the discussion in $(ii)$ that $h$ is the leading order term of $\im\lambda_{\ell,j}$ as $h\to\infty$. Therefore, one can always find $h_{2}\geq h_{1}$ such that for all $h\geq h_{2}$,
		\[
			|\im\lambda_{\ell,j}|\geq\frac{h}{2}.
		\]
		Now, combining this with \cite[Lem.~2]{bogli2025lieb} proves the first restriction of \eqref{eq:lam_ineq}.
		
		Next, we prove the two-sided bound of $|\lambda_{\ell,j}|$. To show $|\lambda_{\ell,j}|\leq(4\pi j)^{2}$, we proceed analogously as in the proof of  \cite[Prop.~20~(ii)]{bogli2025lieb}. It remains to prove $(\pi j)^{2}\leq|\lambda_{\ell,j}|$. With the aid of Lemma \ref{lem:sup_re-im} and 
\eqref{eq:quotientjh} one sees that for all sufficiently large $h$,
		\[
			\left|\frac{\im m}{\Re m}\right|=\frac{|\im m|}{\log\left(4\pi j/\sqrt{h}\right)}\cdot
			\frac{2\pi j}{\Re m}\cdot\frac{\log\left(4\pi j/\sqrt{h}\right)}{2\pi j}
			=\mathcal{O}\left(g(h)h^{\beta-\gamma}\right)
		\]
		for all $m\in B_2(m_{\nu,j}^{(1)})$. Thus $|\im m/\Re m|^{2}<1/2$ for these $m$, so in particular for $m=m_{\nu,j}$. 
		This implies
		\[
			\frac{(\pi j)^{2}}{|\lambda_{\ell,j}|}\leq\frac{(\pi j)^{2}}{|\Re\lambda_{\ell,j}|}=\frac{(\pi j)^{2}}{(\Re m_{\nu,j})^{2}-(\im m_{\nu,j})^{2}}\leq\frac{2(\pi j)^{2}}{(\Re m_{\nu,j})^{2}}.
		\]
		We apply \eqref{eq:sup_re_lim} to the last inequality, which implies that there exists $h_{*}\geq h_{2}$ such that for all $h\geq h_{*}, \ell\in\mathcal{L}_{h}$ and $j\in\mathcal{J}_{h,\ell}$ we get $(2\pi j)^{2}/(\Re m_{\nu,j})^{2}\leq 2$. This completes the proof.
	\end{proof}

\subsection{Proofs of main results}\label{sect:proofsSchr}

As in the previous section, we take constants $0<\beta<\gamma<1/2$ and let $\alpha(h)\in (0,\beta)$ for all $h>0$, with $\alpha(h)\to 0$ so slowly that $h^{-\alpha(h)}\to 0$ as $h\to\infty$. 
We also want to incorporate the constant $0<\varepsilon<1$ that is given in both Theorems \ref{thm: ratio for decreasing func} and \ref{thm: ratio for increasing func}. To this end, we set $\beta=\varepsilon/2$. Then we fix $\gamma$ and restrict $\alpha(h)$ such that
$$0<\alpha(h)<\beta=\frac \varepsilon 2<\gamma<\frac 1 2.$$

Now, again as in the previous section, let us take an arbitrary non-increasing function $g:(0,\infty)\to(0,1]$ such that $g(h)\to 0$ as $h\to\infty$ with, as in~\eqref{eq:gbound},
\[
	g(h)\geq 2h^{\beta-\gamma}
\]
%
for all $h>1$; note that $\beta-\gamma<0$.
Below we will impose further bounds on the decay rates of $g(h)$ and  $h^{-\alpha(h)}$, separately for Theorems \ref{thm: ratio for decreasing func} and \ref{thm: ratio for increasing func}.

We use the  index sets $\mathcal L_h$, $\mathcal J_{h,\ell}$ in \eqref{eq:restr_ell}, \eqref{eq:set_J}. In the proofs below we will need to change the order of the sums over $j$ and $\ell$, hence we also introduce the index sets, for $h>0$,
		\begin{equation}\label{eq: new set J}
			\mathcal{\tilde{J}}_{h}:=\left\{j\in\N : \frac{8h^{\alpha(h)+1/2}}{g(h)}\leq j\leq \frac{h^{\beta+1/2}}{g(h)}\right\},
		\end{equation}
		and, for $h>0$ and $j\in \mathcal{\tilde{J}}_{h}$,
		\begin{equation}\label{eq: new set ell}
			\mathcal{\tilde{L}}_{h,j}:=\left\{\ell\in\N : h^{\alpha(h)+1/2}\leq\ell\leq jg(h)\right\}.
		\end{equation}
		Then, using \eqref{eq:gbound}, it follows that
\begin{equation}\label{eq:indexsets}
\Big\{(j,\ell):\,j\in \mathcal{\tilde{J}}_{h}, \ell\in \mathcal{\tilde{L}}_{h,j} \Big\}
		\subset \Big\{(j,\ell):\,\ell\in\mathcal L_h, j\in\mathcal J_{h,\ell}\Big\}.
\end{equation}
	
	
	Recall that $V_{h}(x)=\ii h\chi_{B_{1}(0)}(x)$ for $x\in\R^{d}$. 
	We have
\begin{equation}\label{eq:VLp}
	\|V_h\|^p_{L^p}=\int_{\mathbb{R}^d}|V_h(x)|^p\,\dd x=\mu_{d}\,h^{p},
\end{equation}
where $\mu_{d}:=\pi^{d/2}/\Gamma(1+d/2)$ is the volume of the unit ball $B_1(0)$ in $\R^{d}$.

	\begin{proof}[Proof of Theorem \ref{thm: ratio for decreasing func}]
Take an arbitrary function $w:[0,\infty)\to [1,\infty)$ with $w(h)\to\infty$ as $h\to\infty$.
Let $c=2^{13}\pi^{2}>0$. 
 We claim that we can choose $g(h)$ and $h^{-\alpha(h)}$ so that
		\begin{equation}\label{eq: assum of g}
		\frac{1}{w(h)} \leq	g^{d-1}(h)\leq 1\qquad\text{and}\qquad g^{d-1}(h)\log\left(\frac{ch^{2\alpha(h)}}{g^{2}(h)}\right)\leq w(h)
		\end{equation}
		for $h>1$.
Indeed, once we impose the restrictions \eqref{eq:gbound} and $\frac{1}{w(h)} \leq g^{d-1}(h)\leq 1$, we see that $g^2(h)\exp(w(h))\to \infty$ as $h\to \infty$. If we choose $h^{-\alpha(h)}$ to decay so slowly that $$(h^{-\alpha(h)})^2 \geq \frac{c}{g^2(h) \exp(w(h))},$$ then also the bound on the right-hand side of \eqref{eq: assum of g} is satisfied.

		 By Proposition \ref{prop:evls_ineq}, there exists $h_{*}\geq 1$ such that for all $h\geq h_{*}$, all $\ell\in\mathcal{L}_{h}$ and all $j\in\mathcal{J}_{h,\ell}$,
		$\lambda_{\ell,j}=\ii h+m_{\nu,j}^{2}$
		is an eigenvalue of the Schr\"{o}dinger operator $H_{V_{h}}$ and, according to \eqref{eq:eigenspace}, an easy calculation \cite[Lem.~21]{bogli2025lieb} shows that its algebraic multiplicity $m(\lambda_{\ell,j})$ satisfies
		\begin{equation}\label{eq: lower bound for multiplicity}
			m(\lambda_{\ell,j})\geq\frac{\ell^{d-2}}{(d-2)!}.
		\end{equation}

		
		Let $h\geq h_{*}$ and let $f:[0,\infty)\to(0,\infty)$ be a continuous, non-increasing function such that
		\[
			\int_{0}^{\infty} f(t)\;\dd t=\infty.
		\]
		Since $\sigma_{\dd}(H_{V_{h}})\subset[0,\infty)+\ii[0,h]$ and with \eqref{eq:VLp} and \eqref{eq:indexsets},
		\begin{align*}
			{\rm Ratio}(V_{h},f)&=
			\frac{1}{\mu_d}\;h^{-p}\sum_{\lambda\in\sigma_\dd(H_{V_{h}})}\dfrac{(\im\lambda)^{p}}{|\lambda|^{d/2}}
			f\left(\log\left(\frac{|\lambda|}{\im\lambda}\right)\right)\\
			&\geq
			\frac{1}{\mu_d}\;h^{-p}\sum_{\ell\in\mathcal{L}_{h}}\sum_{j\in\mathcal{J}_{h,\ell}}
			m(\lambda_{\ell,j})\frac{(\im\lambda_{\ell,j})^{p}}{|\lambda_{\ell,j}|^{d/2}}
			f\left(\log\left(\frac{|\lambda_{\ell,j}|}{\im\lambda_{\ell,j}}\right)\right)\\
			&\geq
			\frac{1}{\mu_d} \;h^{-p}\sum_{j\in\mathcal{\tilde{J}}_{h}}\sum_{\ell\in\mathcal{\tilde{L}}_{h,j}}
			m(\lambda_{\ell,j})\frac{(\im\lambda_{\ell,j})^{p}}{|\lambda_{\ell,j}|^{d/2}}
			f\left(\log\left(\frac{|\lambda_{\ell,j}|}{\im\lambda_{\ell,j}}\right)\right).
		\end{align*}
		Now, we apply \eqref{eq:lam_ineq},\eqref{eq: lower bound for multiplicity}, along with the fact that $f$ is non-increasing, to the last inequality and hence obtain
		\[
			{\rm Ratio}(V_{h},f)\gtrsim
			\sum_{j\in\mathcal{\tilde{J}}_{h}}\sum_{\ell\in\mathcal{\tilde{L}}_{h,j}}
			\frac{\ell^{d-2}}{j^{d}} f\left(\log\left(\frac{2^{5}\pi^{2}j^{2}}{h}\right)\right)
			=\sum_{j\in\mathcal{\tilde{J}}_{h}}\frac{1}{j^{d}}f\left(\log\left(\frac{2^{5}\pi^{2}j^{2}}{h}\right)\right)
			\sum_{\ell\in\mathcal{\tilde{L}}_{h,j}}\ell^{d-2},
		\]
		where the hidden constant depends on $p$ and $d$ only.
		
		Next, we determine a lower bound of $\sum_{\ell\in\mathcal{\tilde{L}}_{h,j}}\ell^{d-2}$. Observe that for any continuous, monotonic function $k:[1,\infty)\to(0,\infty)$,
		\begin{equation}\label{eq: monotone estimate}
			\sum_{x\in\N,\;u\leq x\leq v} k(x)\;\dd x
			\geq\int_{\lceil u\rceil}^{\lfloor v\rfloor} k(x)\;\dd x
			\geq\int_{u+1}^{v-1} k(x)\;\dd x
			\geq\int_{2u}^{v/2} k(x)\;\dd x
		\end{equation}		
		for $u\geq 1$ and $v\geq 2$. Thus, by \eqref{eq: new set ell},
		\[
			\sum_{\ell\in\mathcal{\tilde{L}}_{h,j}}\ell^{d-2}\geq\int_{2h^{\alpha(h)+1/2}}^{jg(h)/2} \ell^{d-2}\;\dd\ell
			=\frac{1}{d-1}\left[\left(\frac{jg(h)}{2}\right)^{d-1}-\left(2h^{\alpha(h)+1/2}\right)^{d-1}\right].
		\]
		Applying the lower bound of $j$ from \eqref{eq: new set J} to the last equality yields
		\[
			\sum_{\ell\in\mathcal{\tilde{L}}_{h,j}}\ell^{d-2}\gtrsim (jg(h))^{d-1};
		\]
		here the non-displayed constant depends on the dimension $d$ only. This implies that
		\begin{equation}\label{eq:double sum j,ell}
			\sum_{j\in\mathcal{\tilde{J}}_{h}}\frac{1}{j^{d}}f\left(\log\left(\frac{2^{5}\pi^{2}j^{2}}{h}\right)\right)
			\sum_{\ell\in\mathcal{\tilde{L}}_{h,j}}\ell^{d-2}
			\gtrsim g^{d-1}(h)\sum_{j\in\mathcal{\tilde{J}}_{h}}\frac{1}{j}f\left(\log\left(\frac{2^{5}\pi^{2}j^{2}}{h}\right)\right).
		\end{equation}
		
		To find a lower bound of the remaining sum over $j\in\mathcal{\tilde{J}}_{h}$, we use the formula \eqref{eq: monotone estimate} one more time, and arrive at
		\begin{align*}
			\sum_{j\in\mathcal{\tilde{J}}_{h}}\frac{1}{j}f\left(\log\left(\frac{2^{5}\pi^{2}j^{2}}{h}\right)\right)
			&\geq
			\int_{2^{4}h^{\alpha(h)+1/2}/g(h)}^{h^{\beta+1/2}/(2g(h))} \frac{1}{j}f\left(\log\left(\frac{2^{5}\pi^{2}j^{2}}{h}\right)\right)\;\dd j\\
			&=
			\int_{\log(ch^{2\alpha(h)}/g^{2}(h))}^{\log(2^{3}\pi^{2}h^{2\beta}/g^{2}(h))} \frac{f(s)}{2} \;\dd s,
		\end{align*}
		where we have made the substitution $s=\log(2^{5}\pi^{2}j^{2}/h)$ and used $c=2^{13}\pi^{2}$. Then
		\[
			\int_{\log(ch^{2\alpha(h)}/g^{2}(h))}^{\log(2^{3}\pi^{2}h^{2\beta}/g^{2}(h))} f(s)\;\dd s
			\geq\int_{\log(ch^{2\alpha(h)}/g^{2}(h))}^{\varepsilon\log h} f(s)\;\dd s;
		\]
		here we have used $\log(2^{3}\pi^{2}h^{2\beta}/g^{2}(h))\geq\log h^{2\beta}$ and $\beta=\varepsilon/2$. Now, recalling the definition $F(x)=\int_{0}^{x} f(t)\;\dd t$, the last lower bound can be rewritten as
		\[
			\int_{\log(ch^{2\alpha(h)}/g^{2}(h))}^{\varepsilon\log h} f(s)\;\dd s=
			F(\varepsilon\log h)-F\left(\log\left(\frac{ch^{2\alpha(h)}}{g^{2}(h)}\right)\right).
		\]
		Since $f$ is non-increasing and positive, we have $F(x)\leq f(0) x$, hence
		\[
			F(\varepsilon\log h)-F\left(\log\left(\frac{ch^{2\alpha(h)}}{g^{2}(h)}\right)\right)
			\geq F(\varepsilon\log h)-f(0)\log\left(\frac{ch^{2\alpha(h)}}{g^{2}(h)}\right).
		\]
		Combining this with \eqref{eq:double sum j,ell} and \eqref{eq: assum of g} implies 
		\[
			{\rm Ratio}(V_{h},f)\gtrsim \frac{F(\varepsilon\log h)}{w(h)}-f(0)w(h);
		\]
		this yields \eqref{eq:Sup(V,f) lower bd decreasing} and completes the proof.
	\end{proof}

	In the following we prove that the appearance of the function $w$ can be removed.
	
	\begin{proof}[Proof of Corollary~\ref{cor:ratio for decreasing func no w}]
		Let $f:[0,\infty)\to(0,\infty)$ be a continuous, non-increasing function with $\int_{0}^{\infty} f(t)\;\dd t=\infty$. Suppose to the contrary that for every $C>0$ and every $h_{*}\geq 1$ there exists $h\geq h_{*}$ such that
		\[
			{\rm Ratio}(V_{h},f)< CF(\varepsilon\log h).
		\]
		
		Let us fix an arbitrary function $w_{0}:[0,\infty)\to [1,\infty)$ with $w_{0}(h)\to\infty$  
		and $w_0^2(h)\ll F(\varepsilon \log h)$
		as $h\to\infty$.
		In view of Theorem~\ref{thm: ratio for decreasing func}, applied with the function $w_0$, and also using~Remark \ref{rem: ratio for decreasing func}, there exists $h_{0}\geq 1$ such that for all $h\geq h_{0}$ one has 
		${\rm Ratio}(V_{h},f)>0$. 
		For $h\geq h_{0}$ we define
		\[
			a_{h}:=\frac{F(\varepsilon\log h)}{{\rm Ratio}(V_{h},f)}>0.
		\]
		Due to the above hypothesis, one can construct stricly increasing sequences $\{h_{n}\}_{n\in\N}$ and $\{a_{h_{n}}\}_{n\in\N}$ such that $a_{h_{1}}>1$,
		\[
			\lim_{n\to\infty}a_{h_{n}}=\infty\quad\text{and}\quad\lim_{n\to\infty} h_{n}=\infty.
		\]
		
		Take an arbitrary non-decreasing function $u:[0,\infty)\to[1,\infty)$ such that $u(h_{n})=a_{h_{n}}$ for all $n\in\N$. Then $u(h)\to\infty$ as $h\to\infty$. At this point, let us choose an arbitrary function $w:[0,\infty)\to [1,\infty)$ with $w(h)\to\infty$ as $h\to\infty$ sufficiently slowly so that
		\[
			w^{2}(h)\ll \min\left\{u(h), F(\varepsilon \log h)\right\},\quad h\to\infty.
		\]
		By Theorem~\ref{thm: ratio for decreasing func} and also~Remark \ref{rem: ratio for decreasing func}, there exists a constant $C_{p,d}>0$ such that for all $h$ sufficiently large one has
		\[
			C_{p,d}\leq w(h)\frac{{\rm Ratio}(V_{h},f)}{F(\varepsilon\log h)}=\frac{w(h)}{a_h}.
		\]
		In particular, one can always find $n_{*}\in\N$ such that for all integers $n\geq n_{*}$
		\[
			C_{p,d}\leq \frac{w(h_n)}{a_{h_n}} \leq \frac{u^{1/2}(h_n)}{a_{h_n}}
			= a_{h_{n}}^{-1/2}.
		\]
		Notice that $a_{h_{n}}^{-1/2}\to 0$ as $n\to\infty$, therefore it yields a contradiction and proves~\eqref{eq:no w}.
	\end{proof}
	
	Now, we prove the divergence rate for non-decreasing functions $f$.
	
	\begin{proof}[Proof of Theorem \ref{thm: ratio for increasing func}]
	Take an arbitrary function $w:[0,\infty)\to [1,\infty)$ with $w(h)\to\infty$ as $h\to\infty$.
	Let $c=2^{8}\pi^{2}>0$.
	Instead of \eqref{eq: assum of g}, we now claim that we can choose $g(h)$ and $h^{-\alpha(h)}$ so that
		\begin{equation}\label{eq:new assum of g}
		\textrm{$h^{-\beta}\ll g^{2}(h)$ as $h\to\infty$};\quad
		\frac{1}{w(h)} \leq	g^{d-1}(h)\leq 1; \quad\text{and} \quad \frac{ch^{2\alpha(h)}}{g^{2}(h)}\leq h^{\beta}
		\end{equation}
		for $h>1$.
 Indeed, once we impose the restrictions \eqref{eq:gbound}, $g^{2}(h)h^{\beta}\to\infty$ as $h\to\infty$ and $\frac{1}{w(h)} \leq	g^{d-1}(h)\leq 1$ for $h>1$, 
		we can choose $h^{-\alpha(h)}$ to decay so slowly that
		$$(h^{-\alpha(h)})^2 \geq \frac{c h^{-\beta}}{g^2(h)}.$$
	
Let  $x_{0}\geq 1$. 
In view of Proposition \ref{prop:evls_ineq}, there exists $h_{*}\geq 1$ such that for all $h\geq h_{*}$, $\ell\in\mathcal{L}_{h}$ and $j\in\mathcal{J}_{h,\ell}$,
		$\lambda_{\ell,j}=\ii h+m_{\nu,j}^{2}$
		is an eigenvalue of $H_{V_{h}}$ with algebraic multiplicity $m(\lambda_{\ell,j})$ satisfying \eqref{eq: lower bound for multiplicity}. Possibly after increasing $h_{*}$, we can assume that $g(h_{*})\leq \pi/x_{0}$. 
		Note that since $g$ is non-increasing, it follows that  $g(h)\leq \pi/x_{0}$ for all $h\geq h_{*}$. Hence for all  $t\geq h^{1/2}/g(h)$,
		\[
			\frac{\pi^{2}t^{2}}{h}\geq x_{0}^{2}.
		\]
		Let $h\geq h_{*}$ and assume that $f:[0,\infty)\to (0,\infty)$ is a continuous, non-decreasing function such that $f(\log t^{2})/t$ is monotonic for $t\geq x_{0}$. Then, by the above observation,
		\[
			t\mapsto\frac{1}{t}f\left(\log\left(\frac{\pi^{2}t^{2}}{h}\right)\right)
		\]
		is also monotonic for all $t\geq h^{1/2}/g(h)$.
	
	We proceed analogously as in the proof of Theorem \ref{thm: ratio for decreasing func} and only point out the differences.	
		Since here we are dealing with a non-decreasing function $f$, instead of an upper bound as in the proof of Theorem \ref{thm: ratio for decreasing func}, we use a lower bound of $|\lambda_{\ell,j}|/\im\lambda_{\ell,j}$ which follows from \eqref{eq:lam_ineq}, namely
		$$\frac{|\lambda_{\ell,j}|}{\im\lambda_{\ell,j}}\geq \frac{\pi^2 j^2}{h}.$$
		 Then, analogously, we arrive at
		\[
			{\rm Ratio}(V_{h},f)\gtrsim
			g^{d-1}(h)\sum_{j\in\mathcal{\tilde{J}}_{h}}\frac{1}{j}f\left(\log\left(\frac{\pi^{2}j^{2}}{h}\right)\right)
			\geq\frac{1}{w(h)}\sum_{j\in\mathcal{\tilde{J}}_{h}}\frac{1}{j}f\left(\log\left(\frac{\pi^{2}j^{2}}{h}\right)\right),
		\]
		where in the last step we have used  the first restriction of \eqref{eq:new assum of g}. It is straightforward to see that \eqref{eq: monotone estimate} can be applied to the remaining sum, therefore, with $c=2^8\pi^2$ and $\beta=\varepsilon/2$,
		\begin{align*}
			\sum_{j\in\mathcal{\tilde{J}}_{h}}\frac{1}{j}f\left(\log\left(\frac{\pi^{2}j^{2}}{h}\right)\right)
			&\geq\int_{\log\left(ch^{2\alpha(h)}/g^{2}(h)\right)}^{\log\left(\pi^{2}h^{2\beta}/(2^{2}g^{2}(h))\right)} \frac{f(s)}{2} \;\dd s\\
			&\geq\frac{1}{2}\left[F(\varepsilon\log h)-F\left(\log\left(\frac{ch^{2\alpha(h)}}{g^{2}(h)}\right)\right)\right]\\
			&\gtrsim F(\varepsilon\log h)-F\left(\frac{\varepsilon}{2}\log h\right),
		\end{align*}
		where we have used the last restriction in \eqref{eq:new assum of g} to get the last lower bound. So, this gives rise to the first inequality of \eqref{eq:Sup(V,f) lower bd increasing}. To argue the second estimate of \eqref{eq:Sup(V,f) lower bd increasing}, let us begin with rewriting
		\[
			F(\varepsilon\log h)-F\left(\frac{\varepsilon}{2}\log h\right)=\int_{(\varepsilon\log h)/2}^{\varepsilon\log h} f(s)\;\dd s.
		\]
		Since $f$ is non-decreasing and positive, we can bound the integral on the right-hand side from below by $(f(0)\varepsilon\log h)/2$ and the proof is complete.
	\end{proof}

	\begin{proof}[Proof of Corollary~\ref{cor:ratio for increasing func no w}]
		The proof primarily relies on the proof idea in Corollary~\ref{cor:ratio for decreasing func no w} presented above, therefore we point out the differences only.
		
		We begin with assuming that the statement is not true and then replace $F(\varepsilon\log h)$ in the proof of Corollary~\ref{cor:ratio for decreasing func no w} by
		\[
			F(\varepsilon\log h)-F\left(\frac{\varepsilon}{2}\log h\right)
		\]
		everywhere except in the construction of the function $w$
		where we only need to require
		\[
			w^{2}(h)\ll u(h),\quad h\to\infty.
		\]
		To demonstrate that there must be a contradiction, it remains to apply Theorem~\ref{thm: ratio for increasing func} instead of Theorem~\ref{thm: ratio for decreasing func} and thus we obtain  the claim.
	\end{proof}
	
	\begin{proof}[Proof of Corollary \ref{cor:exponential growth}]
		Let $\xi>0$ be given. 
		Then for $x>0$
		\[
			F(x)=\int_{0}^{x} e^{\xi t}\;\dd t=\frac{1}{\xi}({\e}^{\xi x}-1),
		\]
		which implies that for $h>0$,
		\[
			F(\varepsilon\log h)-F\left(\frac{\varepsilon}{2}\log h\right)=\frac{h^{\varepsilon\xi}}{\xi}(1-h^{-\varepsilon\xi/2}).
		\]
		Hence, \eqref{eq:Sup(V,f) lower bd for exponential} is obtained by inserting $f(t)={\e}^{\xi t}$ into the first inequality of \eqref{eq:Sup(V,f) lower bd increasing}. Now, it remains to verify that this function meets the requirement of Theorem \ref{thm: ratio for increasing func}. Indeed, for $t\geq 1=:x_{0}$
		\[
			t\mapsto\frac{f(\log t^{2})}{t}=t^{2\xi-1}
		\]
		is monotonic for all $\xi>0$, and now Theorem \ref{thm: ratio for increasing func} implies the claim.
	\end{proof}

Next, in the case of $\int_0^\infty f(x)\dd x=\infty$,  we prove the existence of a potential $V\in L^p(\R^d)$ such that ${\rm Ratio}(V,f)=\infty$.
\begin{proof}[Proof of Theorem~\ref{thm:sum}]
The construction of the potential will rely on the following modification of~\cite[Lem.~2]{bogli2017schrodinger}.
This modification allows for simultaneous approximation of (finitely many) eigenvalues of both $H_{U_1}$ and $H_{U_2}$, in contrast to~\cite[Lem.~2]{bogli2017schrodinger} where only the eigenvalues of one of them could be approximated.
\\
\textbf{Claim 1:} Let $U_1,U_2\in L^p(\R^d)\cap L^\infty(\R^d)$ be decaying at infinity. 
Consider two finite collections of discrete eigenvalues $\lambda_{1,j}\in\sigma_{\dd}(H_{U_1})$, $j=1,\dots,j_1$ ($j_1<\infty$), and $\lambda_{2,j}\in\sigma_{\dd}(H_{U_2})$, $j=1,\dots,j_2$ ($j_2<\infty$).
 Let $x_0\in\R^d\backslash\{0\}$. Then for every $0<\delta<1$ there exist $t_{\delta}>0$ and $r_{\delta}\in(1-\delta,1+\delta)$ 
 such that for all $t\geq t_\delta$ there exist 
 $$\mu_{n,j}(t)\in\sigma_{\dd}(-\Delta+U_1+r_{\delta}^2U_2(r_\delta(\cdot-t x_0)))$$ 
 with $|\mu_{n,j}(t)-\lambda_{n,j}|<\delta$ for $j=1,\dots,j_n$ and $n=1,2$.\\
\textit{Proof of Claim 1.} First note that a scaling argument yields that $r^2\lambda\in\sigma_d(-\Delta+r^2U_2(r\cdot))$ if and only if $\lambda\in\sigma_d(H_{U_2})$.
We want to find a scaling factor $r\in (1-\delta,1+\delta)$ such that 
\begin{equation}\label{eq:avoid}
\begin{aligned}
&\lambda_{1,j}\notin\sigma_{\dd}(-\Delta+r^2U_2(r\cdot)), \quad j=1,\dots, j_1,\\
&r^2\lambda_{2,j}\notin\sigma_{\dd}(-\Delta+U_1), \quad j=1,\dots, j_2. 
\end{aligned}
\end{equation}
Since there are at most finitely many eigenvalues of $\sigma_{\dd}(-\Delta+U_1)$ (resp.\ $\sigma_{\dd}(-\Delta+r^2U_2(r\cdot))$) in a sufficiently small neighbourhood of the unperturbed eigenvalues, which need to be avoided, we can always find a scaling factor $r_\delta:=r$ such that \eqref{eq:avoid} holds. In fact, $|r_{\delta}-1|$ can be arbitrarily small, and we choose it so small that $|r_\delta^2\lambda_{2,j}-\lambda_{2,j}|<\delta/2$ for all $j=1,\dots,j_2$.
Now, to prove Claim 1, we apply~\cite[Lem.~2]{bogli2017schrodinger} to the potentials $U_1$ and $r_{\delta}^2U_2(r_\delta\cdot)$. By applying~\cite[Lem.~2]{bogli2017schrodinger} a second time, with exchanged roles of the potentials, and using
$$\sigma_{\dd}(-\Delta+U_1+r_{\delta}^2U_2(r_\delta(\cdot-tx_0))
=\sigma_{\dd}(-\Delta+U_1(\cdot+tx_0)+r_{\delta}^2U_2(r_\delta\cdot)),$$
 we can approximate the family of  eigenvalues $\lambda_{n,j}$, $j=1,\dots,j_n$ for both $n=1,2$. Note that since there are only finitely many eigenvalues, the parameter $t_\delta$ can be chosen uniformly for all of them. This proves  Claim 1.

Now  we use induction to arrive at the following result which will be used to prove Theorem~\ref{thm:sum}.
This result is similar to \cite[Thm.~1]{bogli2017schrodinger} but in constrast to the latter result, here we only work with the $L^p$ norm, and we do not require that $\sum_{n=1}^\infty \|Q_n\|_{L^\infty}<\infty$.\\
\textbf{Claim 2:} Let $Q_n\in L^\infty(\R^d)$, $n\in\N$, be a family of compactly supported potentials with
$\sum_{n=1}^\infty\|Q_n\|_{L^p}^p<\infty$. 
Given a collection of discrete eigenvalues
$\lambda_{n,j}\in\sigma_{\dd}(H_{Q_n})$, $j=1,\dots,j_n\; (j_n\in\N)$ for $n\in\N$, and given precisions 
$0<\delta_n<1$, $n\in\N$, with $\delta_n<|\im\lambda_{n,j}|$ for all $j=1,\dots,j_n$ and $n\in\N$, we can construct a potential
$$V(x):=\sum_{n=1}^\infty r_n^2Q_n(r_n(x-x_n))$$
with shifts $x_n\in \R^d$ and scaling factors $r_n\in (1-\delta_n,1+\delta_n)$ 
such that the operator $H_V$ has (countably many) discrete eigenvalues $\mu_{n,j}$ with $|\mu_{n,j}-\lambda_{n,j}|<\delta_n$ for all $j=1,\dots, j_n$ and $n\in\N$. The shifts and scaling factors can be chosen such that the $r_n^2Q_n(r_n(\cdot-x_n))$ have disjoint supports and 
$$\int_{\R^d}|V(x)|^p\,\dd x\leq 2\sum_{n=1}^\infty \int_{\R^d}|Q_n(x)|^p\,\dd x<\infty.$$
\textit{Proof of Claim 2.} We follow the idea in the proof of \cite[Thm.~1]{bogli2017schrodinger} and construct the potential $V$ inductively.
  The potential is the limit $N\to\infty$ of 
  $$\tilde V_N(x)=\sum_{n=1}^N r_n^2Q_n(r_n(x-x_n))$$
  (the tilde is used to distinguish from the potentials $V_h$ used in other proofs).
First note that since we can always replace the sequence $\{\delta_{n}\}_{n\in\N}$ by a strictly decreasing sequence $\{\delta'_n\}_{n\in\N}$ converging to $0$ with $\delta'_n\leq\delta_n$ for $n\in\N$, we can assume without loss of generality that our sequence of given precisions is strictly decreasing and converging to $0$.
For the base case $N=1$, let $\tilde V_1=Q_1$. In step $N=2$ we apply Claim 1 with $U_1=\tilde V_1=Q_1$ and $U_2=Q_2$.
In the induction step $N+1$, one applies Claim 1 with $U_1=\tilde V_N$ and $U_2=Q_{N+1}$ to create eigenvalues near $\lambda_{n,j}\in\sigma_{\dd}(H_{Q_{n}})$, $j=1,\dots,j_n$, $n=1,\dots,N+1$, up to a precision $\tilde \delta_{N+1}$ (which plays the role of $\delta$ in  Claim 1).
Note that the constants $\tilde \delta_n$ have to be chosen so small that
$\sum_{n=N}^{\infty} \tilde{\delta_n}\leq \delta_N$ for all $N\in\N$; take for example $\tilde\delta_n=\delta_n-\delta_{n+1}$ and recall that the sequence $\{\delta_n\}_{n\in\N}$ is assumed to be strictly decreasing and converging to $0$.
It is easy to see that the shifts and scaling factors can be chosen such that the potentials $r_n^2Q_n(r_n(\cdot-x_n))$ have disjoint supports and
 $$\int_{\R^d}|V(x)|^p\,\dd x=\sum_{n=1}^\infty \int_{\R^d}|r_n^2Q_n(r_n(x-x_n))|^p\,\dd x\leq 2\sum_{n=1}^\infty \int_{\R^d}|Q_n(x)|^p\,\dd x.$$
 The right-hande side is finite by the assumptions.
 This implies that 
 $$\|V-\tilde V_N\|_{L^p}^p = \sum_{n=N+1}^\infty\int_{\R^d}|r_n^2Q_n(r_n(x-x_n))|^p\,\dd x\to 0$$ as $N\to\infty$.
Now, for every $n\in\N$ and $j=1,\dots,j_n$, for $N\geq n$ there exist $\mu_{N;n,j}\in\sigma_{\dd}(-\Delta+\tilde V_N)$ with $|\mu_{N;n,j}-\lambda_{n,j}|<\delta_n$, i.e.\ the $\mu_{N;n,j}$ are in an $N$-independent disc. Note that the assumption $\delta_n<|\im\lambda_{n,j}|$ implies that this disc does not touch the essential spectrum $[0,\infty)$. Now \cite[Lem.~5.4]{han_11} and its proof implies convergence of the discrete spectrum (including preservation of multiplicities) in each disc as $N\to\infty$. This proves Claim 2.

Now we are ready to  prove Theorem~\ref{thm:sum}. We  choose the potentials $Q_n$ in such a way that $V\in L^p(\R^d)$ but ${\rm Ratio}(V,f)=\infty$.
To this end, let $n_0\in\N$. For $n<n_0$ we take $Q_n\equiv 0$, and for $n\geq n_0$ we take  $Q_n(x)=c_n^2V_n(c_nx)$ with $V_n(x)=\ii n \chi_{B_1(0)}(x)$ and constants $c_n>0$  that will be determined later on; they will also ensure that $\sum_{n=1}^\infty\|Q_n\|_{L^p}^p<\infty$. 
Note that a scaling argument yields that $c_n^2\lambda\in\sigma_d(H_{Q_n})$ if and only if $\lambda\in\sigma_d(H_{V_n})$.

Let $0<\varepsilon<1$ be given. Then Claim 2 (applied with suitably small $\delta_n>0$) and 
the bound~\eqref{eq:rate} (for $d=1$) and Corollary~\ref{cor:ratio for decreasing func no w} (for $d\geq 2$)  imply that, for $n_0>1$ sufficiently large,
	\begin{align*}
	&\sum_{\lambda\in\sigma_\dd(H_{V})}\dfrac{\text{dist}(\lambda,[0,\infty))^{p}}{|\lambda|^{d/2}}f\left(-\log\left(\frac{\text{dist}(\lambda,[0,\infty))}{|\lambda|}\right)\right)\\
	&\gtrsim
	\sum_{n=1}^\infty\;\sum_{\substack{\lambda_{n,j}\in\sigma_\dd(H_{Q_n}) \\ j=1,2,\ldots,j_{n}}}
	\dfrac{\text{dist}(\lambda_{n,j},[0,\infty))^{p}}{|\lambda_{n,j}|^{d/2}}
	f\left(-\log\left(\frac{\text{dist}(\lambda_{n,j},[0,\infty))}{|\lambda_{n,j}|}\right)\right)\\
	&=
	\sum_{n=n_0}^\infty c_n^{2p-d} \sum_{\substack{\lambda_{n,j}\in\sigma_\dd(H_{V_n}) \\ j=1,2,\ldots,j_{n}}}
	\dfrac{\text{dist}(\lambda_{n,j},[0,\infty))^{p}}{|\lambda_{n,j}|^{d/2}}f\left(-\log\left(\frac{\text{dist}(\lambda_{n,j},[0,\infty))}{|\lambda_{n,j}|}\right)\right)\\
	&\gtrsim \sum_{n=n_0}^\infty  c_n^{2p-d}\, F(\varepsilon \log n)\, \int_{\R^d}|V_n(x)|^p\,\dd x;
	\end{align*}
here we have used that the proofs of~\eqref{eq:rate} and Corollary~\ref{cor:ratio for decreasing func no w} only take finitely many discrete eigenvalues of $H_{V_n}$ into consideration.

Claim 2 and a change of variables implies that
$$\int_{\R^d}|V(x)|^p\,\dd x\leq 2\sum_{n=1}^\infty \int_{\R^d}|Q_n(x)|^p\,\dd x
=2\sum_{n=n_0}^\infty c_n^{2p-d}\int_{\R^d}|V_n(x)|^p\,\dd x.$$
By the assumptions on $p$, we have $2p-d>0$. 
Define $$c_n:=\left(\frac{k(\log n)}{n\,\int_{\R^d}|V_n(x)|^p\,\dd x}\right)^{1/(2p-d)}$$
for a continuous, non-increasing function $k:[\log n_0,\infty)\to [0,\infty)$ with $\int_{\log n_0}^\infty k(x)\,\dd x<\infty$; we will explicitly choose this function later on.
Then the above estimates imply
$${\rm Ratio}(V,f)\gtrsim \frac{\sum_{n=n_0}^\infty F(\varepsilon \log n) k(\log n) n^{-1} }{\sum_{n=n_0}^\infty k(\log n) n^{-1}}.$$
Note that $\sum_{n=n_0}^\infty k(\log n)n^{-1}$ is finite by the integral test for convergence, since
$$\int_{n_0}^\infty \frac{k(\log n)}{n}\,\dd n=\int_{\log n_0}^\infty k(x) \,\dd x<\infty.$$
Again using the integral test, in order that ${\rm Ratio}(V,f)=\infty$, it remains to choose the function $k$ such that
\[
	\int_{n_0}^\infty \frac{1}{n}F(\varepsilon \log n) k(\log n)\;\dd n=\int_{\log n_0}^\infty F(\varepsilon x) k(x)\,\dd x=\infty.
\]
Define $K(x):=(F(\varepsilon x))^{-1/2}$ and $k(x):=-K'(x)>0$ for $x\geq \log n_0>0$. Now, the assertion that $k$ is non-increasing for $x\geq\log n_{0}$ can be verified by a direct calculation
\[
	k(x)=\frac{\varepsilon}{2}(F(\varepsilon x))^{-3/2}f(\varepsilon x),
\]
where we have used that $f$ is non-increasing.

Since $\lim_{x\to \infty}K(x)=0$, we see that the assumption $\int_{\log n_0}^\infty k(x)\,\dd x<\infty$ is satisfied.
In addition, for any $x_0\geq \log n_0$, because $F$ is non-decreasing,
$$\int_{\log n_0}^\infty F(\varepsilon x) k(x)\,\dd x \geq F(\varepsilon x_0)\int_{x_0}^\infty k(x)\,\dd x=F(\varepsilon x_0) K(x_0)=(F(\varepsilon x_0))^{1/2}\to\infty$$
as $x_0\to\infty$.
Since the left-hand side is independent of $x_0$, we arrive at $\int_{\log n_0}^\infty F(\varepsilon x) k(x)\,\dd x=\infty$ which concludes the proof.
\end{proof}

Finally, we prove~\eqref{eq:LT cone exponent sharp}, i.e.\ the $\tau$-dependence in \eqref{eq:LT cone} is sharp.

	\begin{proof}[Proof of Theorem~\ref{thm:exponent outside a cone sharp}]
		We take a function with $\varphi(\tau)\ll\tau^{-p}$ as $\tau\to 0$, i.e.\ $\tau^{p}\varphi(\tau)\to 0$ as $\tau\to 0$. We set the parameters similarly as in Section~\ref{Preliminaries} with a modified condition as follows:
		\[
			0<\alpha(h)<\frac{\gamma}{2}<\frac{3\gamma}{4}<\beta<\gamma<\frac{1}{2}.
		\]
		Then one can see that $\beta-\gamma>\alpha(h)-\beta$ for $h>0$. Hence,
		with a function $g$ as in Section~\ref{Preliminaries}, in particular satisfying~\eqref{eq:gbound},
		\[
			g(h)\geq 2h^{\beta-\gamma}\gg\,2h^{\alpha(h)-\beta}
		\]
		as $h\to\infty$. Here note that we need neither~\eqref{eq: assum of g} nor~\eqref{eq:new assum of g}.
		
		Take $\tau>0$ to be $h$-dependent as $\tau(h):=h^{-2\beta}/32\pi^{2}$, which tends to $0$ as $h\to\infty$. 
		We further choose $g$ to decay so slowly that
		\begin{equation}\label{eq:gdecaytau}
			g^{d-1}(h)\gg\frac{\varphi(\tau(h))}{h^{2p\beta}}=(32\pi^{2})^{p}\tau^{p}(h)\varphi(\tau(h)),\quad h\to\infty.
		\end{equation}
		For $h\geq 1$ we consider the Schr\"{o}dinger operator $V_{h}=\ii h\chi_{B_{1}(0)}$ and define the following index sets
		\begin{align*}
			\mathcal{L}'_{h}:=\left\{\ell\in\N : h^{\alpha(h)+1/2}\leq \ell\leq \frac{g(h)}{2}\;h^{\beta+1/2} \right\} \quad\text{and}\quad
			\mathcal{J}'_{h,\ell}:=\left\{j\in\N : \frac{\ell}{g(h)}\leq j\leq h^{\beta+1/2} \right\}.
		\end{align*}
		Notice that $\emptyset\neq\mathcal{L}'_{h}\times\mathcal{J}'_{h,\ell}\subset \mathcal{L}_{h}\times\mathcal{J}_{h,\ell}$ for all $h$ sufficiently large. 
		
		In view of Proposition~\ref{prop:evls_ineq}, there exists $h_{*}\geq 1$ such that for all $h\geq h_{*}$, $\ell\in\mathcal{L}'_{h}$ and $j\in\mathcal{J}'_{h,\ell}$ the number $\lambda_{\ell,j}=\ii h+m_{\nu,j}^{2}$ is an eigenvalue of $H_{V_{h}}$ with the multiplicity $m(\lambda_{\ell,j})\gtrsim \ell^{d-2}$. Due to~\eqref{eq:lam_ineq},
		\[
			\frac{|\Im\lambda_{\ell,j}|}{\Re\lambda_{\ell,j}}\geq\frac{h}{32\pi^{2}j^{2}}\geq\frac{h^{-2\beta}}{32\pi^{2}}=\tau(h).
		\]
		Therefore,
		\begin{align*}
			\frac{1}{\varphi(\tau(h))\|V_{h}\|^{p}_{L^{p}}}\sum_{\substack{\lambda\in\sigma_{\dd}(H_{V_{h}}) \\ |\Im\lambda|\geq \tau(h)\Re\lambda}} |\lambda|^{p-d/2}
			\gtrsim\frac{1}{\varphi(\tau(h))h^{p}}
			\sum_{\ell\in\mathcal{L}'_{h}}\ell^{d-2}\sum_{j\in\mathcal{J}'_{h,\ell}} j^{2p-d}.
		\end{align*}
		We estimate with \eqref{eq: monotone estimate}
		\[
			\sum_{j\in\mathcal{J}'_{h,\ell}} j^{2p-d}\geq \int_{(\ell/g(h))+1}^{(h^{\beta+1/2})-1} j^{2p-d}\;\dd j\gtrsim\;h^{2p\beta+p-(d-1)(\beta+1/2)},
		\]
		and
		\[
			\sum_{\ell\in\mathcal{L}'_{h}}\ell^{d-2}\geq \int_{(h^{\alpha(h)+1/2})+1}^{(g(h)h^{\beta+1/2}/2)-1} \ell^{d-2}\;\dd\ell
			\gtrsim\;g^{d-1}(h)h^{(d-1)(\beta+1/2)},
		\]
		where the above hidden constants depend on $p$ and $d$ only.
		In conclusion, via~\eqref{eq:gdecaytau},
		\[
			\frac{1}{\varphi(\tau(h))\|V_{h}\|^{p}_{L^{p}}}\sum_{\substack{\lambda\in\sigma_{\dd}(H_{V_{h}}) \\ |\Im\lambda|\geq \tau(h)\Re\lambda}} |\lambda|^{p-d/2}
			\gtrsim\;g^{d-1}(h)\frac{h^{2p\beta}}{\varphi(\tau(h))}\to\infty,
		\]
		for $h\to\infty$. This proves~\eqref{eq:LT cone exponent sharp}.
	\end{proof}

		\section{Jacobi operators} \label{sect:jacobi}
		In this section we establish optimal Lieb--Thirring type inequalities for Jacobi operators.

\subsection{Main results}
The proofs of the following main results will be presented in Section~\ref{sect:proofsJac}.

The first result is an improvement of \eqref{eq:H-K bd for Jacobi}.

\begin{theorem}\label{thm:new L-T for non-self Jacobi}
	Let $p\geq 3/2$ and let $f:[0,\infty)\to(0,\infty)$ be a continuous, non-increasing function. If $\int_{0}^{\infty} f(x)\;\dd x<\infty$, then there exists a constant $C_{p,f}>0$ such that for all $v\in \ell^{p}(\Z)$
	\begin{equation}\label{eq:new L-T for non-self Jacobi}
		\sum_{\lambda\in\sigma_{\dd}(J)} 
		\frac{\dist(\lambda,[-2,2])^{p}}{|\lambda^{2}-4|^{1/2}}f\left(-\log\left(\frac{\dist(\lambda,[-2,2])}{\dist(\lambda,\{-2,2\})}\right)\right)
		\leq C_{p,f}\,\|v\|^{p}_{\ell^{p}},
	\end{equation}
	where $C_{p,f}=C_{p}\,\left(\int_{0}^{\infty} f(x)\;\dd x+f(0)\right)$ for an $f$-independent constant $C_p>0$.
\end{theorem}

\begin{remark}
	The bound \eqref{eq:new L-T for non-self Jacobi} reduces to the classical Lieb-Thirring inequality \eqref{eq:L-T for self-adjoint Jacobi} in the self-adjoint case. In addition, this is a generalisation of the Hansmann-Katriel bound \eqref{eq:H-K bd for Jacobi} which is recovered by setting $f(x)=\e^{-\kappa x}$. 
\end{remark}

Next we show that  Theorem \ref{thm:new L-T for non-self Jacobi} is optimal in the sense that if the integrability  condition is removed, then the inequality \eqref{eq:new L-T for non-self Jacobi} cannot be true by proving explicit divergence rates.
To this end, we consider the Jacobi operator $J$ with  $a_k=1, c_k=1$ for all $k\in\Z$, which implies $v_{k}=|b_{k}|$. This is a discrete Schr\"odinger operator with a potential $b$.
For $n\in\N $ let $b=b(n)$ be defined by
\begin{equation}\label{eq:defb}
b_{k}:=
\begin{cases}
	\ii n^{-2/3}\quad&\textrm{if\;\;}k\in\{1,2,\ldots,n\},\\
	0\quad&\textrm{if\;\;}k\in\Z\backslash\{1,2,\ldots,n\}.
\end{cases}
\end{equation}
For ease of notation, we will not explicitly denote the dependence on $n$ by a further index.
Then $b\in \ell^{p}(\Z)$ and an easy calculation shows that
\begin{equation}\label{eq:norm of potential Jacobi}
	\|v\|^{p}_{\ell^{p}}=\|b\|^{p}_{\ell^{p}}=n^{1-2p/3}.
\end{equation}
A numerical range argument \cite[Lem.~4]{bogli2021lieb} implies the inclusion $\sigma_{\dd}(J)\subset[-2,2] + \ii(0,n^{-2/3}]$ for all $n\geq 2$. 

The divergence rates, Theorems~\ref{thm:sharpness thm} and~\ref{thm:sharpness thm for increasing}, that will be formulated below rely on eigenvalue estimates of this type of discrete Schr\"{o}dinger operators above.

\begin{theorem}\label{thm:sharpness thm}
	Let $p\geq 1$ and $\gamma\in(2/3, 1)$. Take a function $g:[1,\infty)\to[1,\infty)$ with $g(n)\to\infty$ as $n\to\infty$ so slowly such that
	\[
		\frac{g(n)}{n^{\gamma-2/3}}\to 0\quad\text{as}\quad n\to\infty.
	\]
	Then there exist $C_{p}>0$ and $n_{*}\geq 2$ such that for all continuous, non-increasing functions $f:[0,\infty)\to(0,\infty)$ with $\int_{0}^{\infty} f(x)\;\dd x=\infty$ and all integers $n\geq n_{*}$
	\begin{equation}\label{eq:sup of new L-T ineq over potential}
		\begin{aligned}
			\sup_{0\neq v\in \ell^{p}(\Z)}\frac{1}{\|v\|^{p}_{\ell^{p}}}\sum_{\lambda\in\sigma_{\dd}(J)} 
			\frac{\dist(\lambda,[-2,2])^{p}}{|\lambda^{2}-4|^{1/2}}&f\left(-\log\left(\frac{\dist(\lambda,[-2,2])}{\dist(\lambda,\{-2,2\})}\right)\right)\\
			&\geq C_{p}\left( F(\log n^{2/3})-3f(0) \log g(n)\right),
		\end{aligned}
	\end{equation}
	where $F(t):=\int_{0}^{t} f(x)\;\dd x$ for $t\geq 0$.
\end{theorem}

\begin{remark}
Clearly, the lower bound on the right-hand side of \eqref{eq:sup of new L-T ineq over potential} diverges provided that $F(\log n^{2/3})$ is divergent faster than $\log g(n)$ as $n\to\infty$. In addition, we note that even though Theorem \ref{thm:new L-T for non-self Jacobi} requires $p\geq 3/2$, here Theorem \ref{thm:sharpness thm} does not.
\end{remark}
	
The following result concerns the divergence rates of the left-hand side of \eqref{eq:sup of new L-T ineq over potential} when a function $f$ is non-decreasing. Obviously, in this case, $\int_{0}^{\infty} f(x)\;\dd x=\infty$.

\begin{theorem}\label{thm:sharpness thm for increasing}
	Let $p\geq 1$,  $0<\varepsilon<2/3<\gamma<1$ and let $x_{0}\geq 1$. Take a function $g:[1,\infty)\to[1,\infty)$ with $g(n)\to\infty$ as $n\to\infty$ so slowly that
	\[
		\frac{g(n)}{n^{\gamma-2/3}}\to 0\quad\text{as}\quad n\to\infty.
	\]
	Then there exist $C_{p}>0$ and $n_{*}\geq 2$ such that for all integers $n\geq n_{*}$ and all continuous, non-decreasing functions $f:[0,\infty)\to(0,\infty)$ such that $f(\log t^{2})/t$ is monotonic for $t\geq x_{0}$ one has
	\begin{equation}\label{eq:sup of new L-T ineq over potential for increasing}
		\begin{aligned}
			\sup_{0\neq v\in \ell^{p}(\Z)}\frac{1}{\|v\|^{p}_{\ell^{p}}}&\sum_{\lambda\in\sigma_{\dd}(J)} 
			\frac{\dist(\lambda,[-2,2])^{p}}{|\lambda^{2}-4|^{1/2}}f\left(-\log\left(\frac{\dist(\lambda,[-2,2])}{\dist(\lambda,\{-2,2\})}\right)\right)\\
			&\geq C_{p}\left[F\left(\log(\pi^{2}n^{\varepsilon})\right)-F\left(\log(\pi^{2}g^{2}(n))\right)\right]
			\geq C_{p}f(0)\;\log\frac{n^{\varepsilon}}{g^{2}(n)}.
		\end{aligned}
	\end{equation}
	\end{theorem}
	
	\begin{remark}
We notice that if $g(n)\ll n^{\varepsilon/2}$, then the right-hand side of~\eqref{eq:sup of new L-T ineq over potential for increasing} diverges as $n\to\infty$.
\end{remark}

\begin{remark}
It would be interesting to investigate whether or not an analogue of Theorem~\ref{thm:sum} for Jacobi operators is true. Unfortunately, the scaling argument that was used in the proof for Schr\"odinger operators is no longer available in the Jacobi case.
\end{remark}


The last main result proves that the $\omega$-dependence of the constant $C_{p,\omega}$ in~\eqref{eq:diamond ineq by G-K}, i.e.\ the order $\tan^{p}(\omega)$ as $\omega\to\frac{\pi}{2}^{-}$, is optimal.

\begin{theorem}\label{thm:sharpconstJacobi}
	Let $p\geq 1$ and let $\varphi:(0,\pi/2)\to(0,\infty)$ be a continuous function such that $\varphi(\omega)\ll\tan^{p}(\omega)$ as $\omega\to\frac{\pi}{2}^{-}$. Then
	\begin{equation}\label{eq:sharp diamond_G-K}
		\limsup_{\omega\to\frac{\pi}{2}^{-}}\sup_{0\neq v\in \ell^{p}(\Z)}\frac{1}{\varphi(\omega)\|v\|^{p}_{\ell^{p}}}
		\sum_{\substack{\lambda\in\sigma_{\dd}(J) \\ 2-\Re\lambda<\tan(\omega)|\Im\lambda|}} 
		|\lambda-2|^{p-1/2}=\infty.
	\end{equation}
\end{theorem}
The proof relies on eigenvalue estimates for the same class of potentials $b$ as in the previous results.

\subsection{Proofs of main results}\label{sect:proofsJac}

First we prove the new Lieb--Thirring type inequalities.

\begin{proof}[Proof of Theorem \ref{thm:new L-T for non-self Jacobi}]
	Assume that the given function $f$ satisfies
	\[
		\int_{0}^{\infty} f(x)\;\dd x<\infty.
	\]
	Let $d>0$. Then, following the construction from \cite[Thm.~2.1]{boegli2023improved}, we can always find a continuous, non-increasing, integrable, piecewise $C^{1}$-function $g:[0,\infty)\to(0,\infty)$ such that 
	\begin{equation}\label{eq:piecewise functions bd}
		f\leq g\quad\text{and}\quad\int_{0}^{\infty} g(x)\;\dd x\leq 2\int_{0}^{\infty} f(x)\;\dd x+\frac{2f(0)}{d}<\infty,
	\end{equation}
and it then follows that, for $a>0$,
	\begin{equation}\label{eq:integral lower bd S}
		\int_{a}^{\infty} \e^{-px}g(x)\;\dd x\geq\frac{1}{p+d}\;\e^{-pa}g(a).
	\end{equation}
	
	Now, we define the following sector in the complex plane
	\[
		\Sigma_{1}:=\{\lambda\in\C : \,\Re\lambda\geq 0,\;2-\Re\lambda<|\Im\lambda|\}
		=\{\lambda\in \Phi^{+}_{\pi/4}:\,\Re\lambda\geq 0\}.
	\]
	It can be seen that for $\lambda\in\Sigma_{1}$, we have $|\lambda+2|\geq 2$ and hence
	\[
		|\lambda-2|^{p-1/2} \geq \dist(\lambda,[-2,2])^{p} \frac{|\lambda+2|^{1/2}}{|\lambda^2-4|^{1/2}}
		\geq\frac{\sqrt{2}}{f(0)}\;\frac{\dist(\lambda,[-2,2])^{p}}{|\lambda^{2}-4|^{1/2}}
		f\left(-\log\left(\frac{\dist(\lambda,[-2,2])}{\dist(\lambda,\{-2,2\})}\right)\right)
	\]
	where we have used that $f(x)\leq f(0)$ for all $x\geq 0$.
	Applying this inequality to \eqref{eq:diamond ineq by G-K} yields
	\begin{equation}\label{eq:L-T on Sigma1}
		\sum_{\lambda\in\sigma_{\dd}(J)\cap\Sigma_{1}} \frac{\dist(\lambda,[-2,2])^{p}}{|\lambda^{2}-4|^{1/2}}
		f\left(-\log\left(\frac{\dist(\lambda,[-2,2])}{\dist(\lambda,\{-2,2\})}\right)\right)\leq C_p\,f(0)\|v\|^{p}_{\ell^{p}}.
	\end{equation}
	Here and in the following, $C_p>0$ denotes a generic constant.
	
	Next, we define another sector $\Sigma_{2}:=\{\lambda\in\C : \Re\lambda\geq 0\}\backslash\Sigma_{1}.$ 
	Note that for $\lambda\in \sigma_d(J)\cap \Sigma_2$ we have 
	$$\Re\lambda\in [0,2),\quad |\Im\lambda|=\dist(\lambda,[-2,2]), \quad \frac{2-\Re\lambda}{|\Im\lambda|}\geq 1.$$
	Due to the inequality \eqref{eq:diamond ineq by G-K} with  $x=\tan(\omega)\in[0,\infty)$, we have the estimate
	\begin{equation}\label{eq:reformulated G-K for Sigma2}
		\sum_{\lambda\in\sigma_{\dd}(J)\cap\Sigma_{2},\;\frac{2-\Re\lambda}{|\Im\lambda|}<x} |\lambda-2|^{p-1/2}\leq C_p\,(1+2x)^{p}\|v\|^{p}_{\ell^{p}}.
	\end{equation}
 We, then, multiply both sides of \eqref{eq:reformulated G-K for Sigma2} by $x^{-p-1}g(\log x)$ and integrate over $x\in[1,\infty)$. For the left-hand side one has
	\begin{align*}
		&\int_{1}^{\infty} x^{-p-1}g(\log x)\sum_{\lambda\in\sigma_{\dd}(J)\cap\Sigma_{2},\;\frac{2-\Re\lambda}{|\Im\lambda|}<x} |\lambda-2|^{p-1/2}\;\dd x\\
		&=\sum_{\lambda\in\sigma_{\dd}(J)\cap\Sigma_{2}} |\lambda-2|^{p-1/2} \int_{(2-\Re\lambda)/|\Im\lambda|}^{\infty} x^{-p-1}g(\log x)\;\dd x\\
		&=\sum_{\lambda\in\sigma_{\dd}(J)\cap\Sigma_{2}} |\lambda-2|^{p-1/2} \int_{\log((2-\Re\lambda)/|\Im\lambda|)}^{\infty} \e^{-px}g(x)\;\dd x\\
		&\geq C_p\,\sum_{\lambda\in\sigma_{\dd}(J)\cap\Sigma_{2}} |\lambda-2|^{p-1/2}\left(\frac{|\Im\lambda|}{2-\Re\lambda}\right)^{p}
			g\left(-\log\left(\frac{|\Im\lambda|}{2-\Re\lambda}\right)\right)\qquad\qquad(\textrm{by \eqref{eq:integral lower bd S}})\\
		&=C_p\,\sum_{\lambda\in\sigma_{\dd}(J)\cap\Sigma_{2}} |\lambda-2|^{p-1/2}\left(\frac{\dist(\lambda,[-2,2])}{2-\Re\lambda}\right)^{p}
		g\left(-\log\left(\frac{\dist(\lambda,[-2,2])}{2-\Re\lambda}\right)\right)\\
		&\geq C_p\,\sum_{\lambda\in\sigma_{\dd}(J)\cap\Sigma_{2}} \frac{\dist(\lambda,[-2,2])^{p}}{|\lambda^{2}-4|^{1/2}}
		g\left(-\log\left(\frac{\dist(\lambda,[-2,2])}{\dist(\lambda,\{-2,2\})}\right)\right),
	\end{align*}
	where we have used $|\lambda-2|^{-1/2}\geq \sqrt{2}\, |\lambda^2-4|^{-1/2}$ and $2-\Re\lambda\leq|\lambda-2|=\dist(\lambda,\{-2,2\})$ in the last step.
	
	For the right-hand side of \eqref{eq:reformulated G-K for Sigma2} one proceeds similarly.
	With $(1+2x)^p\leq 3^p\, x^p$ for $x\geq 1$ we obtain
	\[
		\int_{1}^{\infty} C_p\,(1+2x)^{p}x^{-p-1}g(\log x)\;\dd x\leq 3^{p}C_p\, \int_{0}^{\infty} g(x)\;\dd x.
	\]
	Together with the bounds \eqref{eq:piecewise functions bd}, we finally arrive at
	\begin{equation}\label{eq:L-T on Sigma2}
		\sum_{\lambda\in\sigma_{\dd}(J)\cap\Sigma_{2}} \frac{\dist(\lambda,[-2,2])^{p}}{|\lambda^{2}-4|^{1/2}}
		f\left(-\log\left(\frac{\dist(\lambda,[-2,2])}{\dist(\lambda,\{-2,2\})}\right)\right)\leq C_{p,f}\,\|v\|^{p}_{\ell^{p}}.
	\end{equation}
	Noting that $\Sigma_{1}\cup\Sigma_{2}=\{\lambda\in\C : \Re\lambda\geq 0\}$ we have proven the inequality \eqref{eq:new L-T for non-self Jacobi} for all discrete eigenvalues in the right half-plane by means of \eqref{eq:L-T on Sigma1} and \eqref{eq:L-T on Sigma2}. 
	
	The proof for the left half-plane $\{\lambda\in\C : \Re\lambda\leq 0\}$ is analogous.  Namely, we redefine the  sectors $\Sigma_1, \Sigma_2$ appropriately, 
\begin{align*}
\Sigma_{1}&:=\{\lambda\in\C : \,\Re\lambda\leq 0,\;2+\Re\lambda<|\Im\lambda|\}
		=\{\lambda\in \Phi^{-}_{\pi/4}:\,\Re\lambda\leq 0\}, \\
\Sigma_2&:=\{\lambda\in\C : \Re\lambda\leq 0\}\backslash\Sigma_{1},
\end{align*}	
 and we use
	\[
		\sum_{\lambda\in\sigma_{\dd}(J)\cap\Phi^{-}_{\omega}} |\lambda+2|^{p-1/2}\leq C_{p,\omega}\,\|v\|^{p}_{\ell^{p}}.
	\]
	This completes the proof.
\end{proof}

In order to prove the optimality and divergence rates, we use the following result.
In the following, $J$ always denotes the Jacobi operator with the potential $b$ as in \eqref{eq:defb}.

\begin{proposition}\label{prop:Jacobi}
Let $\gamma\in(2/3,1)$ and let $g$ be a function as in Theorems~\ref{thm:sharpness thm} and~\ref{thm:sharpness thm for increasing}. 
Define
$$\mathcal J(n):=\left\{j\in\Z: \frac 1 2 n^{2/3}g(n)+\frac 3 4\leq j\leq \frac n 8-\frac 1 4\right\}.$$
For $j\in\mathcal J(n)$ define $ x_{j}:=\frac{(4j-1)\pi}{2n}$ and
	\begin{equation}\label{eq:NewRestrictions on angle & radius}
		D_j:=\left\{z=r\e^{\ii \phi}:\,x_{j}-\frac{\pi}{n}\leq\phi\leq x_{j}+\frac{\pi}{n},\; 
		R_1:=1-n^{-\gamma}\leq r\leq 1-\frac{\log g(n)}{n}=:R_2
		\right\}.
	\end{equation}
Then there exists $n_*\in\N$ such that for all integers $n\geq n_*$ and all $j\in\mathcal J(n)$, there exists $z_j=r_j\e^{\ii\phi_j}\in D_j$ such that the operator $J$ has an eigenvalue
$$\lambda_{j}=\ii n^{-2/3}+z_{j}+z^{-1}_{j}=2\cos\phi_j+\ii n^{-2/3}+\mathcal O(n^{-\gamma}), $$
with $\lambda_{j_1}\neq \lambda_{j_2}$ for $j_1\neq j_2$, and
\begin{equation*}
\begin{aligned}
		\dist(\lambda_{j},[-2,2])&=\Im\lambda_{j}=n^{-2/3}+\mathcal{O}(n^{-\gamma}),\\
		\dist(\lambda_{j},\{-2,2\})&=|\lambda_{j}-2|=2(1-\cos\phi_{j})\left(1+\mathcal{O}\left(\frac{1}{g^{2}(n)}\right)\right),\\
		|\lambda^{2}_{j}-4|&=4\sin^{2}\phi_{j}\left(1+\mathcal{O}\left(\frac{1}{g^{2}(n)}\right)\right),
		\end{aligned}
\end{equation*}
	as $ n\to\infty$.
	The involved constants in the error terms are all independent of $j$.
\end{proposition}
	
\begin{proof}
	Due to the eigenvalue construction in \cite[Sect.~2.1]{bogli2021lieb}, the complex solutions $z$ of the polynomial equation
	\begin{equation}\label{eq:charac polynomial eq}
		\ii n^{-2/3}(z^{n+1}-1)(z^{n-1}-1)=z^{n-2}(z^{2}-1)^{2},
	\end{equation}
	with $|z|<1,\;\Im z>0$ and $|z^{n+1}-1|<|z^{n}-z|$, correspond to eigenvalues $\lambda$ of $J$ outside the closed interval $[-2,2]$. In fact, these eigenvalues $\lambda$ are explicitly given by 
	\[
		\lambda=\ii n^{-2/3}+z+z^{-1},
	\]
	see \cite[Prop.~6]{bogli2021lieb}.  

	To solve \eqref{eq:charac polynomial eq} for $z$, we proceed analogously as in \cite[Prop.~8]{bogli2021lieb} with the following modified restrictions. We seek solutions $z=r\e^{\ii\phi}$ of \eqref{eq:charac polynomial eq} in the closed region determined by
	\begin{equation}\label{eq:restrictions on angle & radius}
		n^{-1/3}g(n)\pi\leq\phi\leq\frac{\pi}{4}\quad\text{and}\quad
		1-\frac{1}{\sqrt{n}}\leq r\leq 1-\frac{\log g(n)}{n}.
	\end{equation}
	Note that the assumption on $g$ guarantees that for $n\to\infty$, $n^{-1/3}g(n)\ll n^{\gamma-1}$, therefore the set determined by~\eqref{eq:restrictions on angle & radius} is non-empty for all $n$ sufficiently large.
	First, we prove the existence of solutions of the polynomial equation in this closed region. 
	
	For $n\in\N$ notice that $j\in\mathcal J(n)$ if and only if
	$$
		\left[x_{j}-\frac{\pi}{n},x_{j}+\frac{\pi}{n}\right]\subset\left[n^{-1/3}g(n)\pi,\frac{\pi}{4}\right].
	$$
	We thus observe that for $j\in\mathcal J(n)$, $D_j$ is a subset of that corresponding to \eqref{eq:restrictions on angle & radius}.
	For each $n$ sufficiently large we will employ Rouch\'{e}'s theorem to show that \eqref{eq:charac polynomial eq} has a solution $z$ in the interior of $D_{j}$.
	
	Due to \eqref{eq:restrictions on angle & radius},
	\[
		r^{n}\leq\left(1-\frac{\log g(n)}{n}\right)^{n}=\frac{1}{g(n)}\left(1+\mathcal{O}\left(\frac{\log^{2}g(n)}{n}\right)\right),\quad n\to\infty.
	\]
	Note that $\frac{\log^{2}g(n)}{n}\to 0$ as $n\to\infty$ by the assumptions on $g(n)$.
	Thus
	\begin{equation}\label{eq:radius power n asymp}
		r^{n}=\mathcal{O}\left(\frac{1}{g(n)}\right),\quad n\to\infty.
	\end{equation}
	Moreover, again by \eqref{eq:restrictions on angle & radius},
	\begin{equation}\label{eq:r asymp formula sqrt}
		r=1+\mathcal{O}\left(\frac{1}{\sqrt{n}}\right),\quad n\to\infty.
	\end{equation}

	We rearrange \eqref{eq:charac polynomial eq}:
	\begin{align*}
		&\ii n^{-2/3}(z^{n+1}-1)(z^{n-1}-1)=z^{n-2}(z^{2}-1)^{2}\\
		\Longleftrightarrow\;\;
		&\ii n^{-2/3}+\ii n^{-2/3}(z^{2n}-z^{n+1}-z^{n-1})=z^{n}(z-z^{-1})^{2}\\
		\Longleftrightarrow\;\;
		&\ii n^{-2/3}+4z^{n}\sin^{2}x_{j}=z^{n}[(z-z^{-1})^{2}+4\sin^{2}x_{j}]-\ii n^{-2/3}(z^{2n}-z^{n+1}-z^{n-1}).
	\end{align*}
	Now, define
	\begin{align*}
		&f(z):=\ii n^{-2/3}+4z^{n}\sin^{2}x_{j}\quad\text{and}\\
		&h(z):=\ii n^{-2/3}(z^{2n}-z^{n+1}-z^{n-1})-z^{n}[(z-z^{-1})^{2}+4\sin^{2}x_{j}].
	\end{align*}
	Then both functions $f$ and $h$ are analytic at every $z\in\C\backslash\{0\}\supset D_{j}$. 
	
	For each $n$ sufficiently large it can be verified that 
	\[
		\tilde z_j:=\left(\frac{n^{-2/3}}{4\sin^{2}x_{j}}\right)^{1/n}\e^{\ii x_{j}}
	\]
	is the unique, simple zero of $f(z)$ inside $D_{j}$. 
	Let us check that indeed $\tilde z_j\in D_j$. The condition on the angle is obviously satisfied, so we check the condition on the radius.
	Notice that $\sin x\geq x/2$ for $x\in(0,\pi/2]$, which implies 
	\begin{equation}\label{eq:1/sin}
	\text{for }n^{-1/3}g(n)\pi\leq\phi\leq\frac{\pi}{4}: \quad \frac{1}{2\sin \phi}\leq\frac{1}{\phi}\leq\frac{n^{1/3}}{g(n)\pi}.
	\end{equation}
	This yields, for $\phi=x_j$,
	\begin{equation}\label{eq:quotientnsin}
		\frac{n^{-2/3}}{4\sin^{2}x_{j}}\leq\frac{n^{-2/3}}{x^{2}_{j}}\leq\frac{1}{g^{2}(n)\pi^{2}},
	\end{equation}
	which converges to $0$ as $n\to\infty$. Thus, for all sufficiently large $n$, $\log(n^{-2/3}/4\sin^{2}x_{j})<0$ and
	\[
		\left|\log\frac{n^{-2/3}}{4\sin^{2}x_{j}}\right|=\log(4n^{2/3}\sin^{2}x_{j})\leq\log(4n^{2/3}).
	\]
	Now, we may write
	\[
		|\tilde{z}_{j}|=\exp\left(\frac 1 n \log \frac{n^{-2/3}}{4\sin^{2}x_{j}}\right)
		=1+\frac{1}{n}\log\frac{n^{-2/3}}{4\sin^{2}x_{j}}+\mathcal{O}\left(\frac{\log^{2}n}{n^{2}}\right),\quad n\to\infty.
	\]
	Using \eqref{eq:quotientnsin} we obtain the two-sided estimate
	\[
		\frac{1}{n}\log\frac{n^{-2/3}}{4}\leq\frac{1}{n}\log\frac{n^{-2/3}}{4\sin^{2}x_{j}}
		\leq-\frac{2}{n}(\log g(n)+\log\pi).
	\]
	As a result, the radial condition of \eqref{eq:NewRestrictions on angle & radius} is satisfied, so $\tilde{z}_j\in D_j$. Now Rouch\'{e}'s theorem guarantees that $f(z)+h(z)$ has a unique zero in the interior of $D_{j}$ provided that $|f(z)|>|h(z)|$ for all $z$ on the boundary $\partial D_{j}$.
	
	To prove this, we begin by considering the asymptotic behaviour of $h(z)$ in $D_{j}$ as $n\to\infty$. Since $D_{j}$ is a subregion of that determined by \eqref{eq:restrictions on angle & radius}, one may write
	\[
		\ii n^{-2/3}(z^{2n}-z^{n+1}-z^{n-1})=\mathcal{O}\left(\frac{n^{-2/3}}{g(n)}\right),\quad n\to\infty,
	\]
	where we have used \eqref{eq:radius power n asymp} and \eqref{eq:r asymp formula sqrt}.
	Due to the definition of $D_j$, every $z=r\e^{\ii\phi}\in D_{j}$ satisfies
	\[
		r=1+\mathcal{O}(n^{-\gamma}),\quad n\to\infty,
	\]
	which implies
	\begin{equation}\label{eq:z-z^{-1} asymp in D_{j}}
		z-z^{-1}=r\e^{\ii\phi}-r^{-1}\e^{-\ii\phi}=2\ii\sin\phi+\mathcal{O}(n^{-\gamma}),\quad n\to\infty.
	\end{equation}
	In particular,
	\begin{equation}\label{eq:(z-z^{-1})^{2} asymp}
		(z-z^{-1})^{2}=-4\sin^{2}\phi+\mathcal{O}(n^{-\gamma}),\quad n\to\infty.
	\end{equation}
	Since $\phi\in[x_{j}-\pi/n, x_{j}+\pi/n]$, one has $|\phi-x_{j}|\leq \pi/n$. Furthermore,
	\begin{equation}\label{eq:approx sin x_{j}}
		\sin\phi=\sin(\phi-x_{j})\cos x_{j}+\cos(\phi-x_{j})\sin x_{j}=\sin x_{j}+\mathcal{O}\left(\frac{1}{n}\right),\quad n\to\infty. 
	\end{equation}
	Now, combining \eqref{eq:radius power n asymp}, \eqref{eq:(z-z^{-1})^{2} asymp} and \eqref{eq:approx sin x_{j}} yields
	\[
		z^{n}[(z-z^{-1})^{2}+4\sin^{2}x_{j}]=\mathcal{O}\left(\frac{n^{-\gamma}}{g(n)}\right),\quad n\to\infty.
	\]
	In total, we can infer from $\gamma>2/3$ that, uniformly in $z\in D_{j}$,
	\begin{equation}\label{eq:small asymp Rouche}
		h(z)=\mathcal{O}\left(\frac{n^{-2/3}}{g(n)}\right),\quad n\to\infty.
	\end{equation}
	
	Next, we investigate the asymptotic behaviour of $f(z)$ on the boundary $\partial D_{j}$ as $n\to\infty$. 
	We proceed in three steps. First, using the definition of $D_j$,
	 we consider $z=r\e^{\ii\phi}$ with
	\[
	\phi=x_{j}\pm\frac{\pi}{n}\quad\text{and}\quad R_1\leq r\leq R_2.
	\]
	Here we recall that $x_{j}=(4j-1)\pi/2n$. Consequently,
	\[
	z^{n}=r^{n}\e^{\ii n(x_{j}\pm\pi/n)}=\ii r^{n}.
	\]
	Then
	\[
	|f(z)|=|\ii n^{-2/3}+4\ii r^{n}\sin^{2}x_{j}|\geq n^{-2/3}>|h(z)|
	\]
	for all sufficiently large $n$, 
	where we have taken \eqref{eq:small asymp Rouche} into account.
	
	As a second step, consider $z=r\e^{\ii\phi}$ with $r=R_1$ and $\phi\in[x_{j}-\pi/n, x_{j}+\pi/n]$. It follows that
	\[
	r^{n}=R^{n}_{1}=(1-n^{-\gamma})^{n}=\e^{-n^{1-\gamma}}(1+\mathcal{O}(n^{1-2\gamma})),\quad n\to\infty.
	\]
	Notice that $1-\gamma>0$ while $1-2\gamma<0$. Thus, by the reverse triangle inequality,
	\[
	|f(z)|\geq|n^{-2/3}-4r^{n}\sin^{2}x_{j}|=n^{-2/3}-4 R^{n}_{1}\sin^{2}x_{j}>|h(z)|
	\]
	for all sufficiently large $n$.
	
	As a third and last step, let $z=r\e^{\ii\phi}$ with $r=R_2$ and $\phi\in[x_{j}-\pi/n, x_{j}+\pi/n]$. One readily has
	\[
	r^{n}=R^{n}_{2}=\left(1-\frac{\log g(n)}{n}\right)^{n}
	=\frac{1}{g(n)}\left(1+\mathcal{O}\left(\frac{\log^{2}g(n)}{n}\right)\right),\quad n\to\infty.
	\]
	Hence, with \eqref{eq:1/sin} for $\phi=x_j$,
	\[
		|f(z)|\geq 4 R^{n}_{2}\sin^{2}x_{j}-n^{-2/3}\geq R^{n}_{2}x^{2}_{j}-n^{-2/3}\geq n^{-2/3}g(n)-n^{-2/3}>|h(z)|,
	\]
	for all  $n$ sufficiently large.
	
	Now we can apply Rouch\'{e}'s theorem which proves the existence of a unique solution $z_j$ of \eqref{eq:charac polynomial eq} inside $D_{j}$.

	Next, for each $n$ sufficiently large we prove that the found solutions $z_{j}=r_{j}\e^{\ii\phi_{j}}$,  $j\in\mathcal J(n)$, satisfy the condition $|z^{n+1}_{j}-1|<|z^{n}_{j}-z_{j}|$. Similarly as in \cite[Prop.~7]{bogli2021lieb}, we do this by considering
	\[
		k_{j}:=\frac{1-z^{n+1}_{j}}{z_{j}-z^{n}_{j}},
	\]
	and show that $|k_{j}|<1$. In fact,
	\[
		k_{j}=\frac{1}{z_{j}}\left(1-\frac{z^{n}_{j}(z_{j}-z^{-1}_{j})}{1-z^{n-1}_{j}}\right).
	\]
	Recalling the uniform asymptotic formula \eqref{eq:z-z^{-1} asymp in D_{j}} for $z\in D_{j}$, one can deduce with \eqref{eq:1/sin} for $\phi=\phi_j$,
	\[
		z_{j}-z^{-1}_{j}=2\ii\sin\phi_{j}\left(1+\mathcal{O}\left(\frac{n^{1/3-\gamma}}{g(n)}\right)\right),\quad n\to\infty.
	\] 
	Then, using that $z_{j}$ is a solution of \eqref{eq:charac polynomial eq} and with \eqref{eq:radius power n asymp},
	\[
		\frac{z^{n}_{j}(z_{j}-z^{-1}_{j})}{1-z^{n-1}_{j}}=\ii n^{-2/3}\frac{1-z^{n+1}_{j}}{z_{j}-z^{-1}_{j}}
		=\frac{n^{-2/3}}{2\sin\phi_{j}}\left(1+\mathcal{O}\left(\frac{1}{g(n)}\right)\right),\quad n\to\infty.
	\]		
	Employing $r_j^{-1}=1+\mathcal{O}(n^{-\gamma})$, one arrives at the asymptotic behaviour
	\[
		k_{j}=\e^{-\ii\phi_{j}}\left(1-\frac{n^{-2/3}}{2\sin\phi_{j}}+\mathcal{O}\left(\frac{n^{-2/3}}{g(n)\sin\phi_{j}}\right)\right), \quad n\to\infty.
	\]
	Note that, by the assumption on $g$ and $\gamma>2/3$, one obtains $0<g(n)\sin\phi_{j}\ll n^{\gamma-2/3}$, which implies that $n^{-\gamma}\ll\frac{n^{-2/3}}{g(n)\sin\phi_{j}}$ as $n\to\infty$.
	Thus we can conclude that $|k_{j}|<1$, $j\in\mathcal J(n)$, when $n$ is sufficiently large.
	
	
	In total, for $j\in\mathcal J(n)$ the found solution $z_{j}\in D_j$ gives rise to the eigenvalue
	\[
		\lambda=\lambda_{j}=\ii n^{-2/3}+z_{j}+z^{-1}_{j}.
	\]
	Bearing $\gamma>2/3$ and the definition of $D_j$ in mind, we can write
	\begin{equation}\label{eq:asymptotic for lambda Jacobi}
		\lambda_{j}=2\cos\phi_{j}+\ii n^{-2/3}+\mathcal{O}(n^{-\gamma}),\quad n\to\infty.
	\end{equation}
	Taking the imaginary parts on both sides of \eqref{eq:asymptotic for lambda Jacobi} yields
	$$
		\dist(\lambda_{j},[-2,2])=\Im\lambda_{j}=n^{-2/3}+\mathcal{O}(n^{-\gamma}),\quad n\to\infty.
	$$
	
	Notice that $\cos\phi_{j}\in[1/\sqrt{2},1)$ by \eqref{eq:restrictions on angle & radius}. Consequently, $\Re\lambda_{j}>0$ for $n$ sufficiently large. This readily implies that $\dist(\lambda_{j},\{-2,2\})=|\lambda_{j}-2|$. Since $1-\cos x\geq x^{2}/4$ for $x\in(0,\pi/2]$,
	\[
		\frac{1}{1-\cos\phi_{j}}\leq\frac{4}{\phi^{2}_{j}}\leq\frac{4}{\pi^{2}}\frac{n^{2/3}}{g^{2}(n)},
	\]
	where we have used the first restriction of \eqref{eq:restrictions on angle & radius}. Applying this estimate to \eqref{eq:asymptotic for lambda Jacobi} yields
	$$
		\dist(\lambda_{j},\{-2,2\})=|\lambda_{j}-2|=2(1-\cos\phi_{j})\left(1+\mathcal{O}\left(\frac{1}{g^{2}(n)}\right)\right),\quad n\to\infty.
	$$
	In addition,
	\[
		\lambda^{2}_{j}=4\cos^{2}\phi_{j}+\mathcal{O}(n^{-2/3}),\quad n\to\infty,
	\]
	which results in the asymptotic formula
	\[
		4-\lambda^{2}_{j}=4\sin^{2}\phi_{j}+\mathcal{O}(n^{-2/3}),\quad n\to\infty.
	\]
	Together with \eqref{eq:1/sin} for $\phi=\phi_j$, we arrive at
	$$
		|\lambda^{2}_{j}-4|=4\sin^{2}\phi_{j}\left(1+\mathcal{O}\left(\frac{1}{g^{2}(n)}\right)\right),\quad n\to\infty.
$$
	This concludes the proof.
	\end{proof}
	
	Now, everything is in place for deriving \eqref{eq:sup of new L-T ineq over potential}.

	\begin{proof}[Proof of Theorem \ref{thm:sharpness thm}]
	Suppose that $f:[0,\infty)\to(0,\infty)$ is a continuous, non-increasing function such that 
	\[
		\int_{0}^{\infty} f(x)\;\dd x=\infty.
	\]
	Then it is obvious that $F(t)\to\infty$ as $t\to\infty$.
	
	From \eqref{eq:norm of potential Jacobi}, Proposition~\ref{prop:Jacobi} and the fact that $f$ is non-increasing, we get
	\begin{equation}\label{eq:lower bd for sup of L-T over potential}
	\begin{aligned}
		&\frac{1}{\|v\|^{p}_{\ell^{p}}}\sum_{\lambda\in\sigma_{\dd}(J)} 
		\frac{\dist(\lambda,[-2,2])^{p}}{|\lambda^{2}-4|^{1/2}}f\left(-\log\left(\frac{\dist(\lambda,[-2,2])}{\dist(\lambda,\{-2,2\})}\right)\right)
		\\
		&\geq\frac{1}{n^{1-2p/3}}\sum_{j\in\mathcal J(n)}\frac{n^{-2p/3}}{2^{p}}\frac{1}{4\sin\phi_{j}}
			f\left(\log(8(1-\cos\phi_{j})n^{2/3})\right) \\
		&\geq\frac{1}{2^{p+2}n}\sum_{j\in\mathcal J(n)}\frac{1}{\phi_{j}}f\left(\log(4\phi^{2}_{j}n^{2/3})\right) 		
	\end{aligned}
	\end{equation}
	where we have used $\sin x\leq x$ and $1-\cos x\leq x^{2}/2$ for $x\in(0,\pi/2]$ in the last step.
	
	From the definition of $D_j$, it can be seen that $|\phi_j-x_j|\leq \pi/n$, hence
	\[
		\phi_{j}\leq 2x_{j}=\frac{(4j-1)\pi}{n}=:\theta_{j},
	\]
	for $n$ sufficiently large. Now, applying this to \eqref{eq:lower bd for sup of L-T over potential} yields
	\begin{align*}
		&\frac{1}{2^{p+2}n}\sum_{j\in\mathcal J(n)} f\left(\log(4\phi^{2}_{j}n^{2/3})\right)\\
		&\geq\frac{1}{2^{p+2}n}\sum_{j\in\mathcal J(n)}\frac{1}{\theta_{j}}f\left(\log(4\theta^{2}_{j}n^{2/3})\right)
		\geq\frac{1}{2^{p+2}n}\int_{(2n^{2/3}g(n)+7)/4}^{(n-2)/8} \frac{1}{\theta_{j}}f\left(\log(4\theta^{2}_{j}n^{2/3})\right)\;\dd j\\
		&\geq\frac{1}{2^{p+5}\pi}\int_{4n^{-1/3}g(n)\pi}^{\pi/4} \frac{2}{\theta_{j}}f\left(\log(4\theta^{2}_{j}n^{2/3})\right)\;\dd\theta_{j}
		=\frac{1}{2^{p+5}\pi}\int_{\log(64\pi^{2}g^{2}(n))}^{\log(\pi^{2}n^{2/3}/4)} f(x)\;\dd x,
	\end{align*}
	where we have made the change of variables $x=\log(4\theta^{2}_{j}n^{2/3})$.
	
	At this point, we put $C_{p}:=1/2^{p+5}\pi$ and $n_{*}\geq 2$ can be chosen so large that the above uniform asymptotics hold. Besides, for all $n\geq n_{*}$ we may assume $64\pi^{2}g^{2}(n)\leq g^{3}(n)$.
	Together with $F(x)\leq f(0) x$, we obtain
	\[
		\int_{\log(64\pi^{2}g^{2}(n))}^{\log(\pi^{2}n^{2/3}/4)} f(x)\;\dd x\geq F(\log n^{2/3})-F(\log g^{3}(n))\geq F(\log n^{2/3})-3f(0)\log g(n).
	\]
	Combining this with \eqref{eq:lower bd for sup of L-T over potential} proves \eqref{eq:sup of new L-T ineq over potential} as desired.
\end{proof}

Next, we prove the divergence rate for non-decreasing functions $f$.

%

\begin{proof}[Proof of Theorem~\ref{thm:sharpness thm for increasing}]
	Let $f:[0,\infty)\to(0,\infty)$ be a continuous, non-decreasing function such that $f(\log t^{2})/t$ is monotonic for $t\geq x_{0}$. We consider the Jacobi operator $J$ from Proposition \ref{prop:Jacobi}.
	There exists $n_{0}\geq 2$ such that for all $n\geq n_{0}$ we have $g(n)\geq  2x_{0}/\pi$. In particular, for all $t\geq n^{-1/3}g(n)\pi$
	\[
		\frac{t^{2}n^{2/3}}{4}\geq x^{2}_{0}.
	\]
	Then the function
	\[
		t\mapsto \frac{1}{t}f\left(\log\left(\frac{t^{2}n^{2/3}}{4}\right)\right)
	\]
	is also monotonic for $t\geq n^{-1/3}g(n)\pi$.
	
 Using \eqref{eq:norm of potential Jacobi}, Proposition \ref{prop:Jacobi} and the fact that $1-\cos x\geq x^{2}/4$ for $x\in(0,\pi/2]$ results in
	\begin{align*}
		&\frac{1}{\|v\|^{p}_{\ell^{p}}}\sum_{\lambda\in\sigma_{\dd}(J)} 
		\frac{\dist(\lambda,[-2,2])^{p}}{|\lambda^{2}-4|^{1/2}}f\left(-\log\left(\frac{\dist(\lambda,[-2,2])}{\dist(\lambda,\{-2,2\})}\right)\right)\\
		&\geq\frac{1}{n^{1-2p/3}}\sum_{j\in\mathcal J(n)}\frac{n^{-2p/3}}{2^{p}}\frac{1}{4\sin\phi_{j}}
		f\left(\log((1-\cos\phi_{j})n^{2/3})\right) \\
		&\geq\frac{1}{2^{p+2}n}\sum_{j\in\mathcal J(n)}\frac{1}{\phi_{j}}
		f\left(\log\left(\frac{1}{4}\phi^{2}_{j}n^{2/3}\right)\right).
	\end{align*}
	Recalling $x_{j}=(4j-1)\pi/2n$, one may deduce  from the definition of $D_j$ that
	\[
		\frac{x_{j}}{2}\leq\phi_{j}\leq 2x_{j}
	\]
	 for $n$ sufficiently large. Therefore,
	\begin{align*}
		&\frac{1}{2^{p+2}n}\sum_{j\in\mathcal J(n)}\frac{1}{\phi_{j}}
		f\left(\log\left(\frac{1}{4}\phi^{2}_{j}n^{2/3}\right)\right) 
		\geq
		\frac{1}{2^{p+3}n}\int_{(2n^{2/3}g(n)+7)/4}^{(n-10)/8} \frac{1}{x_{j}}f\left(\log\left(\frac{1}{16}x^{2}_{j}n^{2/3}\right)\right)\;\dd j\\
		&\geq
		\frac{1}{2^{p+5}\pi}\int_{4n^{-1/3}g(n)\pi}^{\pi/8} \frac{2}{x_{j}}f\left(\log\left(\frac{1}{16}x^{2}_{j}n^{2/3}\right)\right)\;\dd x_{j}
		=C_{p}\int_{\log(\pi^{2}g^{2}(n))}^{\log(2^{-10}\pi^{2}n^{2/3})} f(x)\;\dd x,
	\end{align*}
	where we have set $C_{p}=1/2^{p+5}\pi>0$.

Recall that $0<\varepsilon<2/3$.
	Now, we choose $n_{*}\geq n_{0}$ such that for all $n\geq n_{*}$ we get $2^{-10}n^{2/3-\varepsilon}\geq 1$. Hence, one can conclude that
	\[
		\int_{\log(\pi^{2}g^{2}(n))}^{\log(2^{-10}\pi^{2}n^{2/3})} f(x)\;\dd x
		\geq F\left(\log(\pi^{2}n^{\varepsilon})\right)-F\left(\log(\pi^{2}g^{2}(n))\right)
		\geq f(0)\log\frac{n^{\varepsilon}}{g^{2}(n)},
	\]
	which gives rise to \eqref{eq:sup of new L-T ineq over potential for increasing}.
\end{proof}


Finally we prove \eqref{eq:sharp diamond_G-K}.

\begin{proof}[Proof of Theorem~\ref{thm:sharpconstJacobi}]
	First, for $n\in\N$ we let $\omega$ depend on $n$ as
	\[
		\omega(n):=\arctan\left(4(2-\sqrt{2})n^{2/3}\right)\in (0,\pi/2).
	\]
	Due to spectral analysis in Proposition~\ref{prop:Jacobi}, there exists $n_{*}\in\N$ such that for all $n\geq n_{*}$ and all $j\in\mathcal{J}(n)$ the operator $J$ has the eigenvalue $\lambda_j$ inside $[0,2] + \ii(0,n^{-2/3}]$ with
	\[
		|\Im\lambda_{j}|\geq\frac{n^{-2/3}}{2}\quad\text{and}\quad 1-\cos\phi_{j}\leq|\lambda_{j}-2|\leq 4(1-\cos\phi_{j}).
	\]
	Besides, recalling $x_{j}=\frac{(4j-1)\pi}{2n}$,
	\[
		\frac{x_{j}}{2}\leq\phi_{j}\leq x_{j}+\frac{\pi}{n}\leq\frac{\pi}{4}.
	\]
	
	One can see that
	\[
		\frac{2-\Re\lambda_{j}}{|\Im\lambda_{j}|}<\frac{|\lambda_{j}-2|}{|\Im\lambda_{j}|}\leq 8(1-\cos\phi_{j})n^{2/3}
		\leq 4(2-\sqrt{2})n^{2/3}=\tan(\omega(n)),
	\]
	therefore, bearing~\eqref{eq:norm of potential Jacobi} in mind,
	\begin{align*}
		\frac{1}{\varphi(\omega(n))\|v\|^{p}_{\ell^{p}}}
		\sum_{\substack{\lambda\in\sigma_{\dd}(J) \\ 2-\Re\lambda<\tan(\omega(n))|\Im\lambda|}} 
		|\lambda-2|^{p-1/2}
		\geq\frac{1}{\varphi(\omega(n))n^{1-2p/3}}\sum_{j\in\mathcal{J}(n)}(1-\cos\phi_{j})^{p-1/2}.
	\end{align*}
	Noticing that $1-\cos x\geq x^{2}/4$ for $x\in(0,\pi/2]$,
	\begin{align*}
		\sum_{j\in\mathcal{J}(n)}(1-\cos\phi_{j})^{p-1/2}
		\gtrsim\sum_{j\in\mathcal{J}(n)}\phi^{2p-1}_{j} \geq \frac{1}{2^{2p-1}}\int_{n^{2/3}g(n)+3/2}^{(n-2)/16} x^{2p-1}_{j}\;\dd j,
	\end{align*}
	where we have used \eqref{eq: monotone estimate} and $\phi_{j}\geq x_{j}/2$ in the last step.
	
	By change of variable,
	\[
		\int_{n^{2/3}g(n)+3/2}^{(n-2)/16} x^{2p-1}_{j}\;\dd j\geq\frac{n}{2\pi}\int_{2n^{-1/3}g(n)\pi+(5\pi/2n)}^{\pi/8-(3\pi/4n)} x^{2p-1}_{j}\;\dd x_{j}
		\gtrsim n 
	\]
	which yields, using the assumption $\varphi(\omega)\ll \tan^p(\omega)$,
	\begin{align*}
		\frac{1}{\varphi(\omega(n))\|v\|^{p}_{\ell^{p}}}
		\sum_{\substack{\lambda\in\sigma_{\dd}(J) \\ 2-\Re\lambda<\tan(\omega(n))|\Im\lambda|}} 
		|\lambda-2|^{p-1/2}
		\gtrsim\frac{n^{2p/3}}{\varphi(\omega(n))}=\frac{1}{4^{p}(2-\sqrt{2})^{p}}\frac{\tan^{p}(\omega(n))}{\varphi(\omega(n))}\to\infty,
	\end{align*}
	as $n\to\infty$. This proves \eqref{eq:sharp diamond_G-K}. Note that the non-displayed constants depend only on $p$.
\end{proof}

\section*{Acknowledgements}
	The PhD of S.~P.\ is funded through a Development and Promotion of Science and Technology (DPST) scholarship of the Royal Thai Government, ref.\ 5304.1/3758.
	
	\bibliographystyle{acm}
	\bibliography{references}
\end{document}